\documentclass[12pt]{amsart}
\usepackage{latexsym,fancyhdr,amssymb,color,amsmath,amsthm,graphicx,listings,comment}
\usepackage[section]{placeins}
\newtheorem{thm}{Theorem} \newtheorem{propo}{Proposition} 
\newtheorem{lemma}{Lemma}  \newtheorem{coro}{Corollary}
\definecolor{red1}{rgb}{1,0.9,0.9} \definecolor{blue1}{rgb}{0.9,0.9,1} \definecolor{green1}{rgb}{0.9,1,0.9} 
\definecolor{yellow1}{rgb}{1,1,0.9} \definecolor{yellow2}{rgb}{1,1,0.8}

\let\paragraph\subsection

 \newcommand{\ZZ}{\mathbb{Z}}  
   \newcommand{\HH}{\mathbb{H}}

\title{The hydrogen identity for Laplacians}
\author{Oliver Knill}
\date{March 4, 2018}
\address{Department of Mathematics \\ Harvard University \\ Cambridge, MA, 02138 }
\subjclass{11P32, 11R52, 11A41}
\keywords{Graph Laplacian, Spectral radius, Connection Laplacian, graphs}

\begin{document}
\maketitle

\begin{abstract}
For any $1$-dimensional simplicial complex $G$ defined by a finite simple graph, 
the hydrogen identity $|H|=L-L^{-1}$ holds,
where $|H|=(|d|+|d|^*)^2$ is the sign-less Hodge Laplacian defined by the sign-less 
incidence matrix $|d|$ and where $L$ is the connection Laplacian. 
Having linked the Laplacian spectral radius $\rho$ of $G$ with the spectral radius of the adjacency matrix
its connection graph $G'$ allows for every $k$ to estimate $\rho \leq r_k-1/r_k$, where
$r_k=1+(P(k))^{1/k}$ and $P(k)={\rm max}_x P(k,x)$, where $P(k,x)$ is the number of paths of 
length $k$ starting at a vertex $x$ in $G'$. The limit $r_k-1/r_k$ for $k \to \infty$ is
the spectral radius $\rho$ of $|H|$ which by Wielandt is an upper bound for the spectral 
radius $\rho$ of $H=(d+d^*)^2$, with equality if $G$ is bipartite. We can relate so the growth rate
of the random walks in the line graph $G_L$ of $G$ with the one in the connection graph $G'$ of $G$.
The hydrogen identity implies that the random walk $\psi(n) = L^n \psi$ on the connection graph 
$G'$ with integer $n$ solves the $1$-dimensional Jacobi equation $\Delta \psi=|H|^2 \psi$ with 
$\Delta u(n)=u(n+2)-2 u(n)+u(n-2)$ and assures that every solution is represented by such a 
reversible path integral. The hydrogen identity also holds over any finite field $F$. There, the
dynamics $L^n \psi$ with $n \in \ZZ$ is a reversible cellular automaton with alphabet $F^G$. 
By taking products of simplicial complexes, such processes can be defined over any lattice $\ZZ^r$.
Since $L^2$ and $L^{-2}$ are isospectral, by a theorem of Kirby, $L^2$ is always similar to 
a symplectic matrix if the graph has an even number of simplices. 
By the implicit function theorem, the hydrogen relation is robust in the following sense:
any matrix $K$ with the same support than $|H|$ can still 
be written as $K=L-L^{-1}$ with a connection Laplacian satisfying 
$L(x,y)=L^{-1}(x,y)=0$ if $x \cap y =\emptyset$. 
\end{abstract}

\section{Introduction}

\paragraph{}
The largest eigenvalue $\rho$ of the Kirchhoff graph Laplacian $H_0$ of a graph $G$
depends in an intricate way on the vertex degrees of $G$. Quite a few estimates for the spectral
radius $\rho$ are known \cite{GroneMerrisSunder1,GroneMerrisSunder2,Zhang2004,Das2004,
Guo2005,Shi2007,ShiuChan2009, LiShiuChan2010,ZhouXu}. These estimates often use maximal local 
vertex degree quantities which are random walk related. 
We look at the problem by relating the Hodge Laplacian $H=H_0 \oplus H_1$ with
an adjacency matrix $A$ of a connection graph $G'$ defined by $G$. 
This allows to use estimates for the later, like
\cite{BrualdiHoffmann,Stanley1987,Stevanovic} or random cases  \cite{KrivelevichSudakov} 
to get estimates for the graph Laplacian. 
The adjacency spectral radius $r(G)$ is accessible through random walk estimates and therefore 
has monotonicity properties like $r(A) \leq r(B)$ if $A$ is a subgraph of $B$. While
the spectrum of the Kirchoff Laplacian $H_0$ has already been analyzed with the help of the line graph,
an idea going back at least to \cite{AndersonMorely1985}, 
the hydrogen relation $L-L^{-1}=|H|$ considered here is an other link. 

\paragraph{}
Adjacency matrices are more pleasant for estimates than Laplacians because they are 
$0-1$ matrices and so non-negative. As the spectral radius of an adjacency matrix $A$ is a 
Lyapunov exponent which can be estimated as $\rho(A) \leq \rho(A^k)^{1/k} \leq (P(k))^{1/k}$,
where $P(k)$ is the number of paths of length $k$ in the corresponding graph, one can
get various estimates by picking a fixed $k$ and go about estimating the maximal number of paths of length
$k$ starting at some point. Already the simplest estimates for $A$ leads to effective estimate 
$\rho$, especially for Barycentric refinements, where $\rho=(d+3)-(d+3)^{-1}$ where $d$ is 
the maximal vertex degree. In general, we get in the simplest case $\rho \leq r-1/r$, 
where $r$ is 1 plus the maximal of the sum of a pair vertex degrees of adjacent vertices. 
Global estimates like \cite{BrualdiHoffmann,Stanley1987} for adjacency matrices lead 
to other upper bounds. Because the spectrum of $L$ has negative parts, the 
Schur inequality discussed below can be useful too. It produces insight about the spectrum of 
$L$ and so about the spectrum of $|H|$. 

\paragraph{}
For a finite abstract simplicial complex $G$ - a finite set of non-empty sets closed 
under the operation of taking non-empty subsets - the connection Laplacian $L$ is
defined by the properties $L(x,y)=1$ if $x \cap y$ intersect and $L(x,y)=0$ if $x \cap y =\emptyset$.
The matrix $L$ is the Fredholm matrix $L={\bf 1}+A$ for the adjacency matrix $A$ of the connection
graph $G'$ and has remarkable properties. First of all, it is always unimodular. Then, the sum of the 
matrix entries of the inverse $L^{-1}$ is the Euler characteristic $\chi(G)$ of $G$. One can also
write $\chi(G)=p(G)-n(G)$, where $p(G)$ are the positive and $n(G)$ are the 
negative eigenvalues of $L$ \cite{HearingEulerCharacteristic}. Finally, the matrix entries 
$g(x,y)$ of the inverse $g=L^{-1}$ can be given
explicitly as $\omega(x) \omega(y) \chi(St(x) \cap St(y))$, where 
${\rm St}(x) = \{ y \supset x \; | \; y \in G\}$ is the star of $x$ 
and $\omega(x)=(-1)^{{\rm dim}(x)}$ \cite{ListeningCohomology}.

\paragraph{}
In the case of a 1-dimensional complex, $G$ enjoys more properties:
the spectrum of the square $L^2$ satisfies the symmetry $\sigma(L^2)=1/\sigma(L^2)$.
This relation indicates already that $L$ is of symplectic nature. It is equivalent to a
functional equation $\zeta(s)=\zeta(-s)$ for the zeta function associated to $L$ \cite{DyadicRiemann}.
The hydrogen operator $L-L^{-1}$ was of interest to us at first as its trace is 
a Dehn-Sommerville type quantity $\sum_{x} \chi(S(x))$ \cite{DehnSommerville}, 
where $S(x)$ is the unit sphere of a vertex $x$ in the Barycentric refinement graph. It is zero 
for nice triangulations of even dimensional manifolds. The main point we make here is that
in one dimensions, $L-L^{-1}$ always is the sign-less Hodge operator $|H|=(|d|+|d|)^*$, where
$|d|$ is the sign-less gradient. For bipartite graphs, the sign-less Hodge operator $|H|$
is conjugated to the usual Hodge operator $H$. 
In \cite{HearingEulerCharacteristic}, we have looked at the hydrogen relation 
only in the case of graphs which were Barycentric refinements and so bipartite which 
implied $|H|$ to be conjugated to $H$. 

\section{The hydrogen relation}

\paragraph{}
The Kirchhoff Laplacian $H_0$ of a graph $G$ with $v$ vertices is the $v \times v$ matrix $H=B-A$, 
where $B$ is the diagonal vertex degree matrix and $A$ is the adjacency matrix of $G$.
The relation $d_0f( (x,y) )=f(y)-f(x)$ defines an incidence matrix $d$ which is defined, once an 
orientation has been chosen on the edge set $E$. The choice of orientation does not affect
$H$, nor the spectrum of $D=d+d^*$. The matrix $d_0$ is a $e \times v$ matrix, where
$e$ is the number of edges of $G$. It defines naturally a $(e+v) \times (e+v)$ matrix $d$ and so
$D=d+d^*$ with $D^2=H$. The matrix
$H_0$ can now be rewritten as $H_0=d_0^* d_0$. It is so the discrete analogue of 
$\Delta = {\rm div} \circ {\rm grad}$ in calculus. If we look at the sign-less derivative 
$|d_0|$, which is a 0-1 matrix,  then still $|d_0|^2=0$ in one dimensions and $
|H|=(|d|+|d|^*)^2=|H_0| \oplus |H_1|$ is the sign-less Laplacian. Unlike $H_0$, which always 
has an eigenvalue $0$, the sign-less Kirchhoff matrix $|H_0|$ is invertible if the graph $G$ is not bipartite. 

\paragraph{}
The Barycentric refinement $G_1$ of a finite simple graph $G=(V,E)$ is a new graph $\Gamma$ with 
vertex set $V_1=V \cup E$ and where a pair $(a,b) \in V \times E$ 
is in the edge set $E_1$ if $a \subset b$. If $v$ is the cardinality of $V$ and $e$ is the 
cardinality of $E$, then $\chi(G)=v-e$ is the Euler characteristic of $G$.
The new graph $\Gamma$ has $n=v+e$ vertices and $m=\sum_{v} {\rm deg}(v) = 2e$ 
edges. The invariance of Euler characteristic under Barycentric refinements is in $1$-dimensions just 
the Euler handshake formula. The Euler characteristic $\chi(G)$ of a simplicial complex $G$ is in 
general (not necessarily $1$-dimensional case) an invariant as the $f$-vector $f=(v,e,\dots)$ 
of a complex is explicitly mapped to a new vector $S f$ for the matrix $S$ which has 
$(1,-1,1,\dots)$ as eigenvalues of $A^T$. The matrix $S$ is upper triangular and given by
$S_{ij} = i! S(j,i)$, where $S(j,i)$ are {\bf Stirling numbers} of the second kind.

\paragraph{}
The $1$-form Laplacian $H_1 = d d^*$ defined by the $v \times e$ matrix $d$ has the same non-zero spectrum
than $H_0 = d^* d$. This super symmetry relation between the 1-form Laplacian $H_1$
and 0-form Laplacian $H_0$ appears have been used first in \cite{AndersonMorely1985}.
The nullity of $H_0$ is the genus $b_1$ while the nullity of $H_0$ is $b_0$ the 
number of connected components of $G$. The Euler characteristic $\chi=v-e$ satisfies the Euler-Poincar\'e
formula $b_0-b_1$. The Dirac matrix $D=d+d^*$ is a $n \times n$ matrix with $n=v+e$. It defines
$H=D^2$. Because $d^2=(d^*)^2=0$, we have $D^2=d^* d + d d^*$ so that
$H=H_0 \oplus H_1$ decomposes into two blocks. If we take the absolute values of all
entries, we end up with the sign-less Hodge Laplacian $|H|=(|d|+|d|^*)^2=|d| |d|^* + |d|^* |d|$. 

\paragraph{}
The connection matrix of $G$ is $L(x,y) = 1$ if $x \cap y \neq \emptyset$
and $L(x,y)=0$ if $x \cap y = \emptyset$. For $1$-dimensional complexes $G$,
$L$ is a $(v+e) \times (v+e)$ matrix. It has the same size than 
the Hodge operator $H$ or its sign-less Hodge matrix $|H|$.
But unlike $H$ or $|H|$, the matrix $L$ is always invertible and its inverse 
$L^{-1}$ is always integer-valued. Also, unlike the Hodge matrix $H$, which is reducible, the matrix $L$ is 
always irreducible if $G$ is connected. The reason for this ergodicity property is that $L$ is a 
non-negative matrix and that $L^n$ is a positive matrix for large enough $n$ as one can see when looking
at a random walk interpretation: the matrix is the adjacency matrix of the graph graph $G'$ for which 
loops have been attached to each vertex $x$. 

\begin{thm}[Hydrogen relation]
For any finite simple graph, $|H|=L-L^{-1}$.
\end{thm}
\begin{proof}
While we have already seen this in \cite{ListeningCohomology}, we have looked there only
at the situation when $H$ and $|H|$ were similar. The proof of the hydrogen relation
is in the $1$-dimensional case especially simple, as we can explicitly write down the
entries of the inverse $L^{-1}$. We have $(L-L^{-1})(x)  = \chi(S(x))= {\rm deg}(x)$
where ${\rm deg}(x)$ is the vertex degree of the Barycentric refinement. The 
Euler characteristic $\chi(S(x))$ of the unit sphere $S(x)$ is equal to the vertex degree 
of an original vertex $x$ and equal to $2$ for an edge, because an edge only has two neighbors 
provided that $G$ is $1$-dimensional.
If $x$ is vertex and $y$ an edge containing $x$, then $L(x,y)=1$ and 
$L^{-1}(x,y)=1$. If $x,y$ are vertices, then $L(x,y)=0$ and $L^{-1}(x,y)=-1$. 
If $x,y$ are edges, then $L(x,y)=1$ and $L^{-1}(x,y)=0$. 
\end{proof}

\paragraph{}
The name ``hydrogen" had originally been chosen because the Laplacian $L=-\Delta/(4\pi)$ in Euclidean
space has an inverse with a kernel $1/|x-y|$ of the potential of the hydrogen atom. 
While the Laplacian in $R^3$ is not invertible, the integral operator $L^{-1}$ allows formally 
to see $L-L^{-1}$ as a quantum mechanical system, where at each point of space, a ``charge" is 
located. We can think of the Laplacian $L$ itself as a kinetic, and the ``integral operator" $L^{-1}$ as a
potential theoretic component of $H$. This is a reason also why for the invertible connection
Laplacian $L$, the inverse entries $g(x,y)$, the Green's functions, has an
interpretation as a potential energy between two simplices $x,y$. 

\paragraph{}
The energy theorem 
$$   \sum_{x} \sum_y g(x,y) = \chi(G) $$
which holds for an arbitrary simplicial complex, then suggests to see the Euler characteristic 
as a ``total potential energy". The kinetic operator $L$ is unitary conjugated to an operator $M$ for which 
the sum $\sum_x \sum_y M(x,y) = \omega(G)$ is the Wu characteristic. As it involves pairs
of intersecting simplices, the Wu characteristic can be seen as a 
potential theoretic part and Euler characteristic as kinetic energy. The hydrogen relation combines the kinetic
and potential theoretic part. For locally Euclidean complexes, for which Poincar\'e duality
holds, the two energies balance out and the total energy is zero. For structures
with boundary, the energy (the curvature of the valuation) is supported on the boundary, similarly 
as the charge equilibria of Riesz measures are supported on the boundary of regions in potential theory. 

\paragraph{}
Since for every connected graph, the operator $L$ is irreducible with a unique
Perron-Frobenius eigenvalue, this is also inherited by the sign-less Hodge operator $|H|$. 
The corresponding eigenvalue has no sign changes. Now, the maximal eigenvalue can 
only become smaller if $|H|$ is replaced by $H$. 

\begin{coro}[Relating Hodge and Adjacency] 
The spectral radius $\rho$ of $H$ is bounded above by $r-1/r$, where
$r$ is the spectral radius of $L$. 
\end{coro}
\begin{proof}
From the hydrogen relation, we only get the spectral radius $\rho(|H|) = r-1/r$. 
But as $|H|$ is a dominating matrix for $H$, a result of Wielandt 
(\cite{MincNonnegative}, Chapter 2.2, Theorem 2.1),
assures that the maximal eigenvalue of $H$ is smaller or equal than the maximal
eigenvalue of $|H|$.
\end{proof}

\paragraph{}
Since $L=1+A$ with adjacency matrix $A$ for $G'$, we can estimate $\rho$ in terms of the eigenvalues of $A$.
$L$ is a non-negative matrix as it only has entries $0$ or $1$.
If $G$ is a Barycentric refinement of a complex, we can estimate the 
largest eigenvalue of $L$ as the maximal row sum. In the case of a refined graph
where every edge is connected to a vertex with only 2 neighbors, we get the
two vertex degree $d+2$, where $d$ is the maximal vertex degree of $G$. Therefore, 
the connection Laplacian spectral radius can be estimated from above by $d+3$. 
If $G$ is a Barycentric refinement we have then
$$  \rho \leq (d+3) - 1/(d+3) \; . $$
We can do better by incorporating also neighboring
vertex degrees. Still, already here, the estimate in \cite{AndersonMorely1985} is better,
as the maximal row sum in the 1-form Laplacian is then $2+(d_x-1)+1=d_x+2$.
The Anderson-Morely estimates gives for Barycentric refined graphs
$$  \rho \leq d+2 \; . $$

\paragraph{}
{\bf Examples:} \\
{\bf 1)} For a star graph $S(n)$, where a single vertex has degree $n$, the
spectral radius of $H$ is $n+1$. For $n=4$, where the $S(4)$ has spectral
radius $5$, the Barycentric refinement has spectral radius $5.30278$. 
The estimate $d+3 - 1/(d+3) = 7-1/7=6.85714$ is less effective than the estimate
using neighboring vertex degrees. \\
{\bf 2)} For the circular graph $C(2n)$ the spectral radius is $4$ for all $n$. 
As $2n$ is even, they are Barycentric refinements. The above estimate gives
$5-1/5=4.8$ independent of $n$. This is a regular graph for which the estimate 
in \cite{LiShiuChan2010} does not apply. \\
{\bf 3)} For a linear graph $L(2n+1)$ we always get the estimate $5-1/5$, the same
estimate as for circular graph. The spectral radius is $3$ for $n=1$ and 
increases monotonically to $4$ for $n \to \infty$. For $n=1$, the 
dual vertex spectral estimate gives $3.75$ which is better than the
Brualdi-Hoffmann-Stanley estimate 3.4314 or \cite{LiShiuChan2010} which gives $3.9333$.
However, for larger $n>1$, the estimate \cite{LiShiuChan2010} is the best.  
There are estimates $2d-2/((2l+1) n)$ (\cite{Shi2007} Theorem 3.5)
or  $2d-1/(l n)$ (\cite{LiShiuChan2010} Theorem 2.3) for irregular graphs of diameter
$l$ and $n$ vertices. We see that there are cases, where the estimate does not
even beat the trivial estimate $2d$.   \\
{\bf 4)} For the complete graphs $K_2,K_3$, the global estimates 
\cite{BrualdiHoffmann,Stanley1987} are best. 

\paragraph{}
We also get:

\begin{coro}[Adjacency estimate] 
For any finite simple graph $G=(V,E)$, the spectral radius estimate $\rho \leq 1+r-1/(1+r)$ holds,
where $r$ is the maximum eigenvalue the adjacency matrix of $G'$. 
\label{dualvertex}
\end{coro}
\begin{proof}
This follows immediately from the hydrogen relation. 
\end{proof} 

\paragraph{}
The simplest estimate is when $r$ is the maximum of $d_x+d_y$ over all pairs $(x,y) \in E$. 
This is a bit higher than the Anderson-Morely estimate \cite{AndersonMorely1985},
who gave $\rho \leq r$ with $r={\rm max}_{(x,y) \in E} (d_x +d_y)$. Their 
proof of 1985 which only  used the maximal row sum in the 1-form Laplacian. 
(That article acknowledges H.P. McKean for posing the problem and uses that 
$d^* d$ and $d d^*$ are essentially isospectral, which is a special case of 
McKean-Singer super symmetry). In \cite{LiZhang1997}, the upper bound 
$2+\sqrt{(d_x-2) (d_y-2)}$ is given and \cite{Zhang2004} state $r= {\rm max} (d_x + \sqrt{d_x m_x})$,
where $m_x$ is the average neighboring degree of $x$. More estimates are given in \cite{FengLiZhang}. 
We can improve such bound by looking at paths of length $2$. But such
estimates are not pretty. The statement allows for any estimate on the maximal eigenvalue of the
adjacency matrix $A$ of $G'$ gives and a spectral estimate of the Laplacian $H_0$. 
Better is the estimate $r \leq 1+{\rm max}_i (1/r_i) \sum_j A_{ij} r_j$, where 
$r_j$ is the row sum of the adjacency matrix which is bounded above by $d_x+d_y$.

\paragraph{}
Global estimates are still often better in highly saturated graphs, where
pairs of vertices where the maximal vertex degree exist. The estimates through $L$ are good 
if the maximum is attained at not-adjacent places. Considerable work was done 
already to improve the trivial upper bound $\rho \leq 2d$ which just estimates the row rums
of $|H|$ and so gives a bound for $\rho(H)$.
In principle, one can get arbitrary close to the spectral radius of $|H|$. 
Let $P(k,x)$ denote the number of paths of length $k$ in the connection graph $G'$ 
starting at $x$. The spectral radius of $L$ can be estimated by 
$r_k = 1+{\rm max}_x P(k,x)$. 

\begin{coro}[Random walk estimate]
For any $k$, $\rho \leq r_k - 1/r_k$. 
\end{coro}

\paragraph{}
One can so estimate the spectral radius of $|H|$ arbitrary well. To make it effective
and geometric, one would have to estimate $P(k,x)$ in terms of the local geometry of the graph. 
In the bipartite case, these estimates become sharp for $H$ even at least in the limit
$k \to \infty$. In any case, we see that we can estimate the spectral radius
of a graph Laplacian of a graph $G$ dynamically by the growth rate of random walks 
in a related connection graph $G'$. It would be nice to exploit this
for Erd\"os-Renyi graphs. It is reasonable as the upper bound estimates depend on the clustering of
a large number of large vertex degrees, leading to a high number $P(k,x)$ of paths
of length $k$ starting from a point $x$. 

\paragraph{}
If $G$ is connected, the maximal eigenvalue of the sign-less Kirchhoff matrix
$|H_0|$ is unique. The reason is that $|H_0|$ is a non-negative matrix which 
for some power $m$ satisfies is a positive matrix $|H_0|^m$ having 
a unique maximal eigenvalue by the Perron-Frobenius theorem. The classical
Kirchhoff matrix $H_0$ 
itself can have multiple maximal eigenvalues. For the complete graph $K_n$ for example,
the maximal eigenvalue $n$ of $H_0$ appears with multiplicity $n-1$, while
$|H_0|$ has only a single maximal eigenvalue $2n-2$ and a single eigenvalue $0$.
All other $n-1$ eigenvalues are equal to $n-2$. 

\begin{center} \begin{figure}
\scalebox{0.92}{\includegraphics{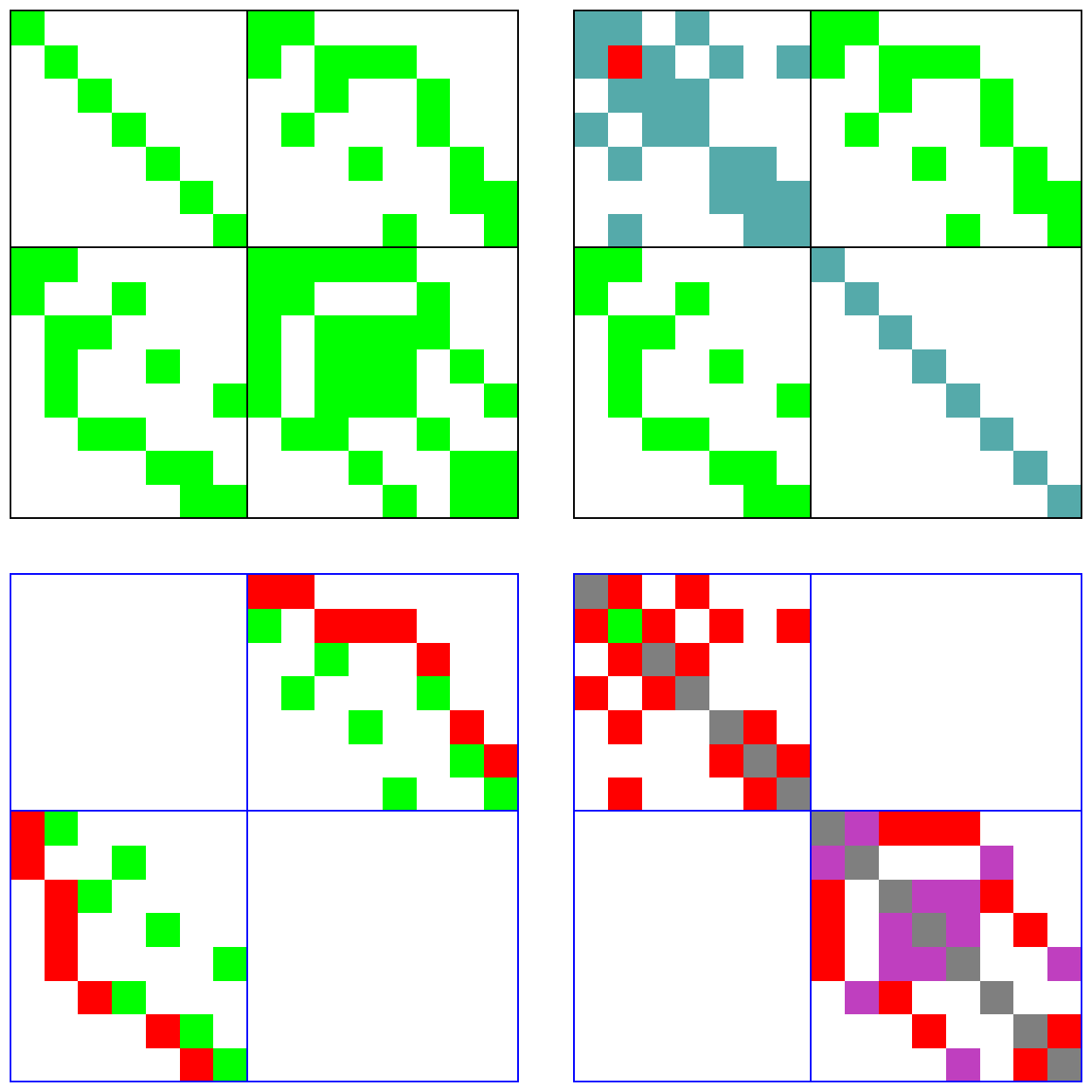}} \\
\caption{
\label{matrices}
We see the connection Laplacian $L$, its inverse $g=L^{-1}$, 
the Dirac operator $D=d+d^*$ and the Hodge Laplacian $H=D^2$
for the Barycentric refined figure $8$-graph. The hydrogen
identity is $|H|=L-g$. In this case we know that $H$ is isospectral
to $L-g$. 
}
\end{figure} \end{center}

\paragraph{}
Here are the matrices written out. 
\begin{tiny}
$$ L = \left[
\begin{array}{ccccccccccccccc}
 1 & 0 & 0 & 0 & 0 & 0 & 0 & 1 & 1 & 0 & 0 & 0 & 0 & 0 & 0 \\
 0 & 1 & 0 & 0 & 0 & 0 & 0 & 1 & 0 & 1 & 1 & 1 & 0 & 0 & 0 \\
 0 & 0 & 1 & 0 & 0 & 0 & 0 & 0 & 0 & 1 & 0 & 0 & 1 & 0 & 0 \\
 0 & 0 & 0 & 1 & 0 & 0 & 0 & 0 & 1 & 0 & 0 & 0 & 1 & 0 & 0 \\
 0 & 0 & 0 & 0 & 1 & 0 & 0 & 0 & 0 & 0 & 1 & 0 & 0 & 1 & 0 \\
 0 & 0 & 0 & 0 & 0 & 1 & 0 & 0 & 0 & 0 & 0 & 0 & 0 & 1 & 1 \\
 0 & 0 & 0 & 0 & 0 & 0 & 1 & 0 & 0 & 0 & 0 & 1 & 0 & 0 & 1 \\
 1 & 1 & 0 & 0 & 0 & 0 & 0 & 1 & 1 & 1 & 1 & 1 & 0 & 0 & 0 \\
 1 & 0 & 0 & 1 & 0 & 0 & 0 & 1 & 1 & 0 & 0 & 0 & 1 & 0 & 0 \\
 0 & 1 & 1 & 0 & 0 & 0 & 0 & 1 & 0 & 1 & 1 & 1 & 1 & 0 & 0 \\
 0 & 1 & 0 & 0 & 1 & 0 & 0 & 1 & 0 & 1 & 1 & 1 & 0 & 1 & 0 \\
 0 & 1 & 0 & 0 & 0 & 0 & 1 & 1 & 0 & 1 & 1 & 1 & 0 & 0 & 1 \\
 0 & 0 & 1 & 1 & 0 & 0 & 0 & 0 & 1 & 1 & 0 & 0 & 1 & 0 & 0 \\
 0 & 0 & 0 & 0 & 1 & 1 & 0 & 0 & 0 & 0 & 1 & 0 & 0 & 1 & 1 \\
 0 & 0 & 0 & 0 & 0 & 1 & 1 & 0 & 0 & 0 & 0 & 1 & 0 & 1 & 1 \\
\end{array} \right] $$

$$ L^{-1} = \left[
\begin{array}{ccccccccccccccc}
 -1 & -1 & 0 & -1 & 0 & 0 & 0 & 1 & 1 & 0 & 0 & 0 & 0 & 0 & 0 \\
 -1 & -3 & -1 & 0 & -1 & 0 & -1 & 1 & 0 & 1 & 1 & 1 & 0 & 0 & 0 \\
 0 & -1 & -1 & -1 & 0 & 0 & 0 & 0 & 0 & 1 & 0 & 0 & 1 & 0 & 0 \\
 -1 & 0 & -1 & -1 & 0 & 0 & 0 & 0 & 1 & 0 & 0 & 0 & 1 & 0 & 0 \\
 0 & -1 & 0 & 0 & -1 & -1 & 0 & 0 & 0 & 0 & 1 & 0 & 0 & 1 & 0 \\
 0 & 0 & 0 & 0 & -1 & -1 & -1 & 0 & 0 & 0 & 0 & 0 & 0 & 1 & 1 \\
 0 & -1 & 0 & 0 & 0 & -1 & -1 & 0 & 0 & 0 & 0 & 1 & 0 & 0 & 1 \\
 1 & 1 & 0 & 0 & 0 & 0 & 0 & -1 & 0 & 0 & 0 & 0 & 0 & 0 & 0 \\
 1 & 0 & 0 & 1 & 0 & 0 & 0 & 0 & -1 & 0 & 0 & 0 & 0 & 0 & 0 \\
 0 & 1 & 1 & 0 & 0 & 0 & 0 & 0 & 0 & -1 & 0 & 0 & 0 & 0 & 0 \\
 0 & 1 & 0 & 0 & 1 & 0 & 0 & 0 & 0 & 0 & -1 & 0 & 0 & 0 & 0 \\
 0 & 1 & 0 & 0 & 0 & 0 & 1 & 0 & 0 & 0 & 0 & -1 & 0 & 0 & 0 \\
 0 & 0 & 1 & 1 & 0 & 0 & 0 & 0 & 0 & 0 & 0 & 0 & -1 & 0 & 0 \\
 0 & 0 & 0 & 0 & 1 & 1 & 0 & 0 & 0 & 0 & 0 & 0 & 0 & -1 & 0 \\
 0 & 0 & 0 & 0 & 0 & 1 & 1 & 0 & 0 & 0 & 0 & 0 & 0 & 0 & -1 \\
\end{array} \right] $$

$$ D = \left[
\begin{array}{ccccccccccccccc}
 0 & 0 & 0 & 0 & 0 & 0 & 0 & -1 & -1 & 0 & 0 & 0 & 0 & 0 & 0 \\
 0 & 0 & 0 & 0 & 0 & 0 & 0 & 1 & 0 & -1 & -1 & -1 & 0 & 0 & 0 \\
 0 & 0 & 0 & 0 & 0 & 0 & 0 & 0 & 0 & 1 & 0 & 0 & -1 & 0 & 0 \\
 0 & 0 & 0 & 0 & 0 & 0 & 0 & 0 & 1 & 0 & 0 & 0 & 1 & 0 & 0 \\
 0 & 0 & 0 & 0 & 0 & 0 & 0 & 0 & 0 & 0 & 1 & 0 & 0 & -1 & 0 \\
 0 & 0 & 0 & 0 & 0 & 0 & 0 & 0 & 0 & 0 & 0 & 0 & 0 & 1 & -1 \\
 0 & 0 & 0 & 0 & 0 & 0 & 0 & 0 & 0 & 0 & 0 & 1 & 0 & 0 & 1 \\
 -1 & 1 & 0 & 0 & 0 & 0 & 0 & 0 & 0 & 0 & 0 & 0 & 0 & 0 & 0 \\
 -1 & 0 & 0 & 1 & 0 & 0 & 0 & 0 & 0 & 0 & 0 & 0 & 0 & 0 & 0 \\
 0 & -1 & 1 & 0 & 0 & 0 & 0 & 0 & 0 & 0 & 0 & 0 & 0 & 0 & 0 \\
 0 & -1 & 0 & 0 & 1 & 0 & 0 & 0 & 0 & 0 & 0 & 0 & 0 & 0 & 0 \\
 0 & -1 & 0 & 0 & 0 & 0 & 1 & 0 & 0 & 0 & 0 & 0 & 0 & 0 & 0 \\
 0 & 0 & -1 & 1 & 0 & 0 & 0 & 0 & 0 & 0 & 0 & 0 & 0 & 0 & 0 \\
 0 & 0 & 0 & 0 & -1 & 1 & 0 & 0 & 0 & 0 & 0 & 0 & 0 & 0 & 0 \\
 0 & 0 & 0 & 0 & 0 & -1 & 1 & 0 & 0 & 0 & 0 & 0 & 0 & 0 & 0 \\
\end{array}
\right] $$

$$ H = D^2 =  \left[
\begin{array}{ccccccccccccccc}
 2 & -1 & 0 & -1 & 0 & 0 & 0 & 0 & 0 & 0 & 0 & 0 & 0 & 0 & 0 \\
 -1 & 4 & -1 & 0 & -1 & 0 & -1 & 0 & 0 & 0 & 0 & 0 & 0 & 0 & 0 \\
 0 & -1 & 2 & -1 & 0 & 0 & 0 & 0 & 0 & 0 & 0 & 0 & 0 & 0 & 0 \\
 -1 & 0 & -1 & 2 & 0 & 0 & 0 & 0 & 0 & 0 & 0 & 0 & 0 & 0 & 0 \\
 0 & -1 & 0 & 0 & 2 & -1 & 0 & 0 & 0 & 0 & 0 & 0 & 0 & 0 & 0 \\
 0 & 0 & 0 & 0 & -1 & 2 & -1 & 0 & 0 & 0 & 0 & 0 & 0 & 0 & 0 \\
 0 & -1 & 0 & 0 & 0 & -1 & 2 & 0 & 0 & 0 & 0 & 0 & 0 & 0 & 0 \\
 0 & 0 & 0 & 0 & 0 & 0 & 0 & 2 & 1 & -1 & -1 & -1 & 0 & 0 & 0 \\
 0 & 0 & 0 & 0 & 0 & 0 & 0 & 1 & 2 & 0 & 0 & 0 & 1 & 0 & 0 \\
 0 & 0 & 0 & 0 & 0 & 0 & 0 & -1 & 0 & 2 & 1 & 1 & -1 & 0 & 0 \\
 0 & 0 & 0 & 0 & 0 & 0 & 0 & -1 & 0 & 1 & 2 & 1 & 0 & -1 & 0 \\
 0 & 0 & 0 & 0 & 0 & 0 & 0 & -1 & 0 & 1 & 1 & 2 & 0 & 0 & 1 \\
 0 & 0 & 0 & 0 & 0 & 0 & 0 & 0 & 1 & -1 & 0 & 0 & 2 & 0 & 0 \\
 0 & 0 & 0 & 0 & 0 & 0 & 0 & 0 & 0 & 0 & -1 & 0 & 0 & 2 & -1 \\
 0 & 0 & 0 & 0 & 0 & 0 & 0 & 0 & 0 & 0 & 0 & 1 & 0 & -1 & 2 \\
\end{array} \right] $$

$$ |H|=L-L^{-1} =  \left[ \begin{array}{ccccccccccccccc}
                   2 & 1 & 0 & 1 & 0 & 0 & 0 & 0 & 0 & 0 & 0 & 0 & 0 & 0 & 0 \\
                   1 & 4 & 1 & 0 & 1 & 0 & 1 & 0 & 0 & 0 & 0 & 0 & 0 & 0 & 0 \\
                   0 & 1 & 2 & 1 & 0 & 0 & 0 & 0 & 0 & 0 & 0 & 0 & 0 & 0 & 0 \\
                   1 & 0 & 1 & 2 & 0 & 0 & 0 & 0 & 0 & 0 & 0 & 0 & 0 & 0 & 0 \\
                   0 & 1 & 0 & 0 & 2 & 1 & 0 & 0 & 0 & 0 & 0 & 0 & 0 & 0 & 0 \\
                   0 & 0 & 0 & 0 & 1 & 2 & 1 & 0 & 0 & 0 & 0 & 0 & 0 & 0 & 0 \\
                   0 & 1 & 0 & 0 & 0 & 1 & 2 & 0 & 0 & 0 & 0 & 0 & 0 & 0 & 0 \\
                   0 & 0 & 0 & 0 & 0 & 0 & 0 & 2 & 1 & 0 & 0 & 0 & 1 & 0 & 0 \\
                   0 & 0 & 0 & 0 & 0 & 0 & 0 & 1 & 2 & 1 & 1 & 1 & 0 & 0 & 0 \\
                   0 & 0 & 0 & 0 & 0 & 0 & 0 & 0 & 1 & 2 & 1 & 1 & 0 & 0 & 1 \\
                   0 & 0 & 0 & 0 & 0 & 0 & 0 & 0 & 1 & 1 & 2 & 1 & 1 & 0 & 0 \\
                   0 & 0 & 0 & 0 & 0 & 0 & 0 & 0 & 1 & 1 & 1 & 2 & 0 & 1 & 0 \\
                   0 & 0 & 0 & 0 & 0 & 0 & 0 & 1 & 0 & 0 & 1 & 0 & 2 & 0 & 0 \\
                   0 & 0 & 0 & 0 & 0 & 0 & 0 & 0 & 0 & 0 & 0 & 1 & 0 & 2 & 1 \\
                   0 & 0 & 0 & 0 & 0 & 0 & 0 & 0 & 0 & 1 & 0 & 0 & 0 & 1 & 2 \\
                  \end{array} \right] $$ 
\end{tiny}

\paragraph{}
There are more relations if $|H|$ is similar to $H$ \cite{ListeningCohomology}. 
In particular, the eigenvalues $0$ of $H$ which belong
to the harmonic 1-forms lead to eigenvalues $-1$ of $L$ and the eigenvalues $0$ of
$H$ which belong to harmonic 0-forms lead to eigenvalues $1$ of $L$ \cite{ListeningCohomology}.

\paragraph{}
The relation could also be useful to analyze the distribution of the Laplacian
eigenvalues, especially in the upper part of the spectrum. 
We can apply the Schur inequality for $L$ to get
$$  \sum_{i=1}^t \lambda_i  \leq t$$
for the ordered list of eigenvalues $\lambda_i$ of $L$. As eigenvalues of $L$ are first negative
this is a suboptimal lower bound for small $t$. But as there is an equality for $t=n$, there will be
a compensation in the upper part. This gives lower bounds on the largest eigenvalues of $L$.
By the way, also the Fiedler inequality giving the lower bound of $d \leq \delta(H) \leq 2 d$
follows directly from Schur. But for $L$ Schur gives have $n-1 \leq \delta(L) \leq n$.

\begin{center} \begin{figure}
\scalebox{0.48}{\includegraphics{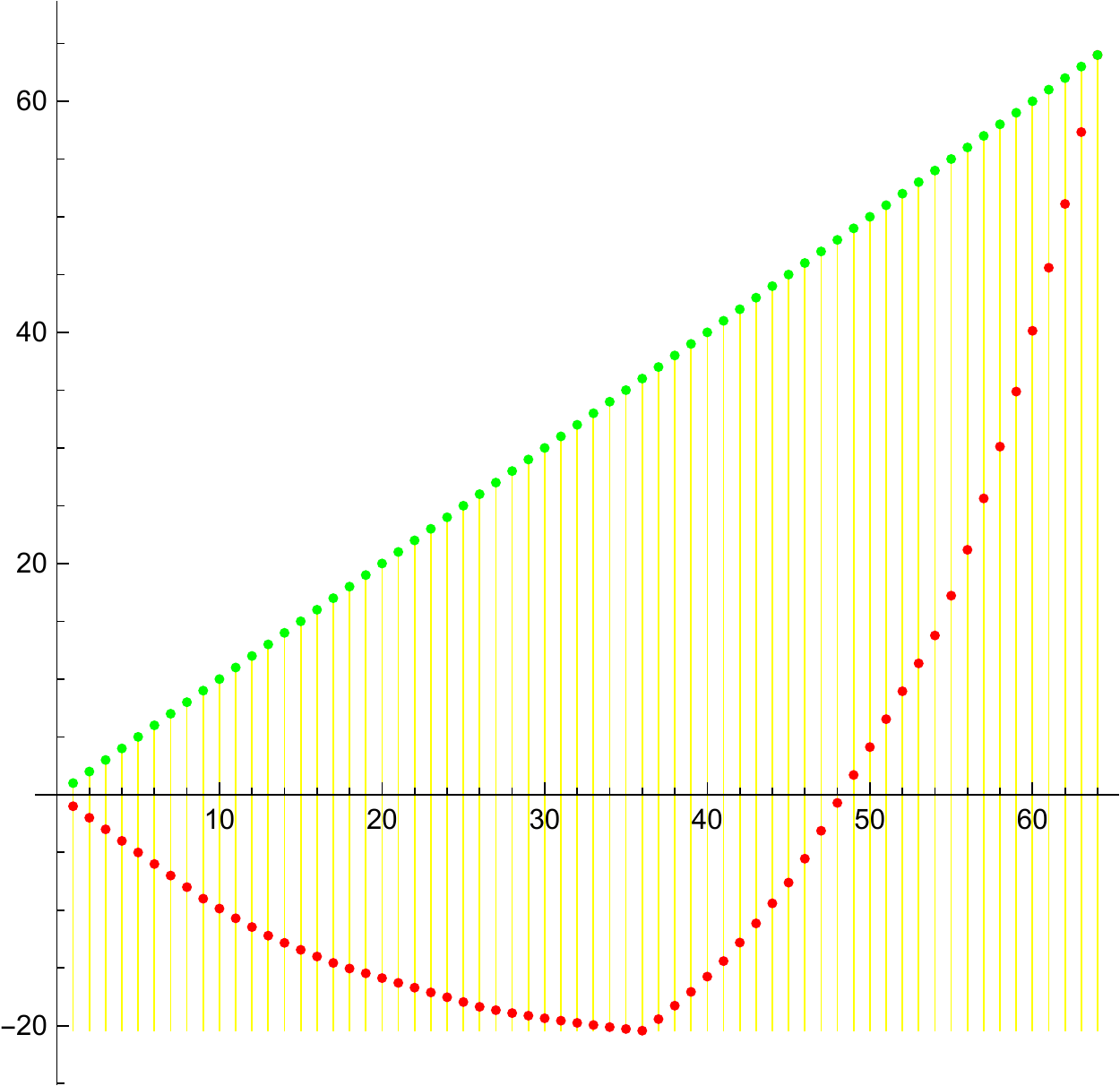}}
\scalebox{0.48}{\includegraphics{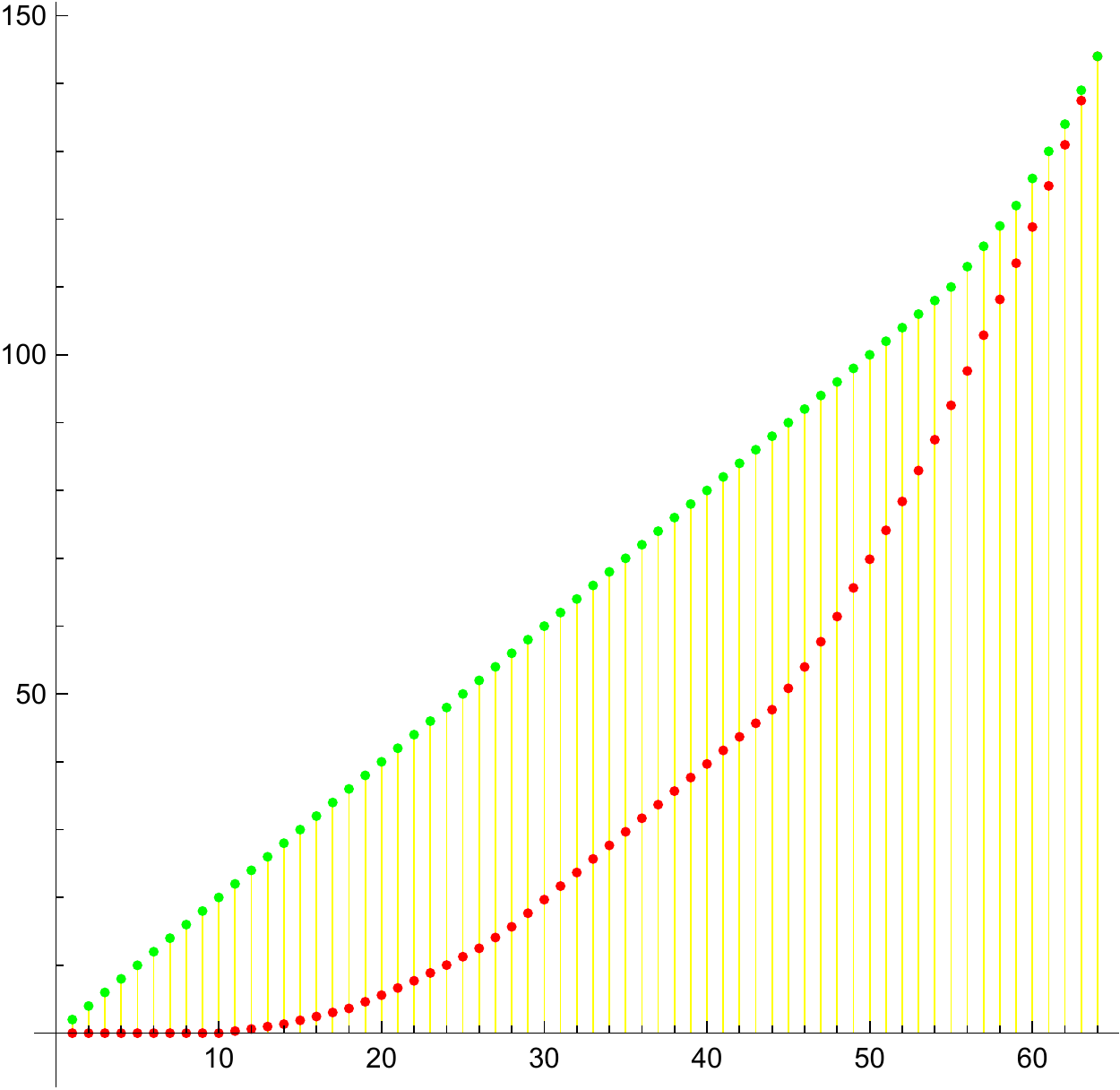}}
\caption{
\label{graphs}
An illustration of the Schur inequality first for $L$ and then for 
$|H|=L-L^{-1}$. The inequality allows for example to estimate the gap 
between the largest $\lambda_n$ and second largest eigenvalue $\lambda_{n-1}$.
The simplest estimate is $\lambda_n(L)-\lambda_{n-1}(L) \geq 1$. 
For $|H|$, Schur gives the Fiedler inequality $\lambda_n(|H|) \geq d$, 
where $d$ is the largest vertex degree. Better estimates can be obtained by 
taking powers $L^k$. One can the estimate the largest gap 
$\lambda_n(L)-\lambda_{n-1}(L) \geq P(k)^{1/k}$
from the maximal number $P(k)$ of closed paths of length $k$ in $G'$.
}
\end{figure} \end{center}

\paragraph{}
Here is quadratic relation which illustrates the relation with a random walk: 
we know $||\lambda||_2^2 = \sum_{i=1} \lambda_i^2 = \sum_{i,j} L_{ij}^2={\rm tr}(L^2) = 
\sum_{i,j} L_{ij}$ to hold for any simplicial complex $G$. In the 
$1$-dimensional case, the eigenvalues of $L^2$ are the eigenvalues of $L^{-2}$ so
that the eigenvalues $\lambda$ of $L$ satisfy 
$||\lambda||_2^2 = ||\lambda^{-1}||_2^2={\rm tr}(L^{-2})$. So,
$$  {\rm tr}(|H|^2) = {\rm tr}((L-L^{-1})^2) = 2 {\rm tr}(L^2)-2n 2||\lambda||_2^2  -2n  \; . $$
As $|H|=|H_0| \oplus |H_1|$ and $|H_0|=|d_0|^* |d_0|,H_1=|d_0| |d_0|^*$ are essentially isospectral,
the left hand side is 
$$ 2 {\rm tr}(|H_1|^2)  = \sum_{x} {\rm deg}'(x) = 4 |E'| \; , $$
where ${\rm deg}'(x)$ is the vertex degree of the connection graph of $G$ and $|E'|$ is 
the number of edges in the connection graph $G'$. As in the $1$-dimensional case, the
Euler characteristic satisfies $\chi(G)=\chi(G')$ as $G'$ is homotopic to $G$ (note that $G'$ has
triangles and is no more $1$-dimensional), the
energy theorem tells $\chi(G)=\sum_{i,j} L^{-1}_{i,j}$. 

\begin{center} \begin{figure}
\scalebox{0.22}{\includegraphics{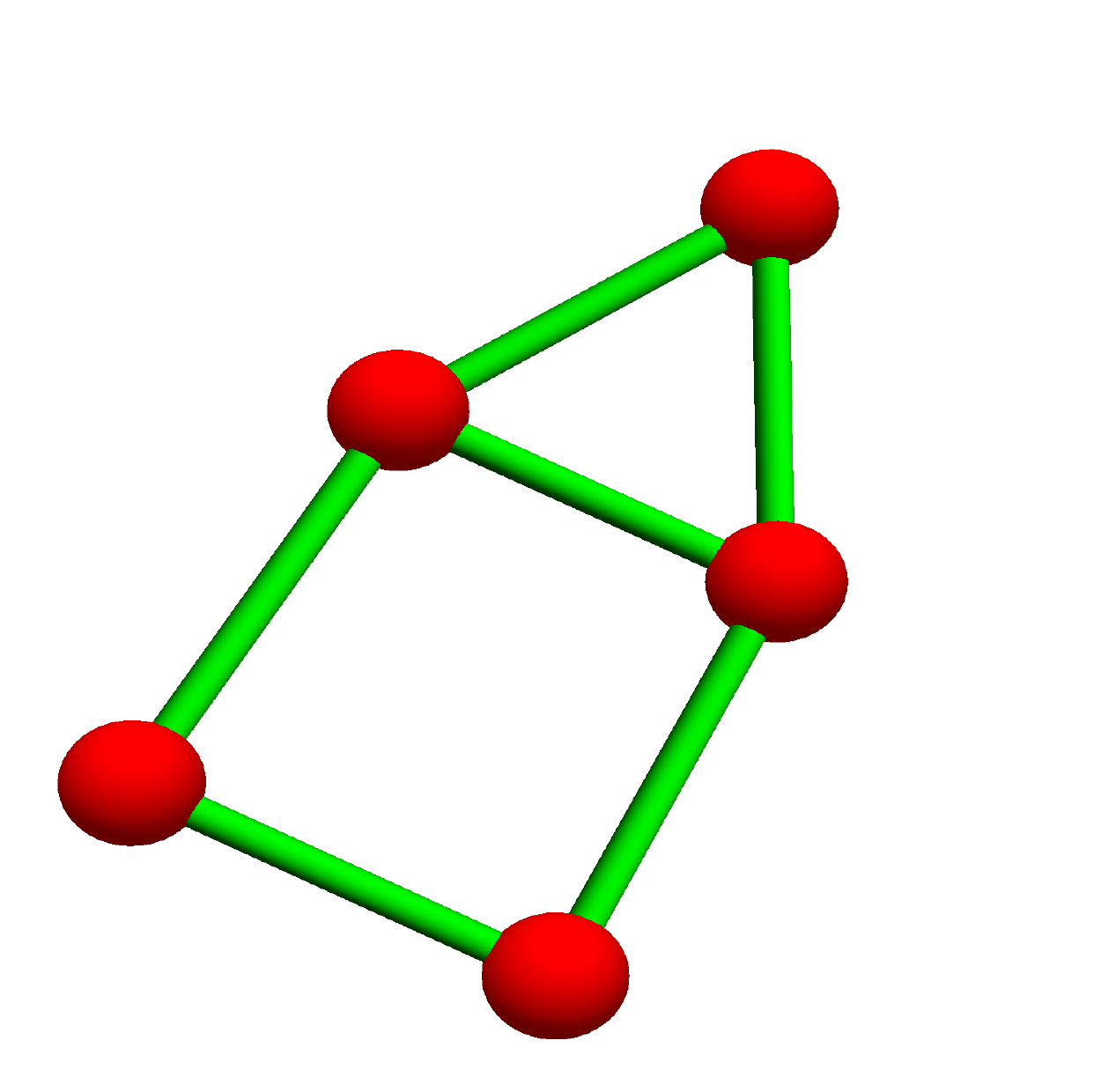}} 
\scalebox{0.22}{\includegraphics{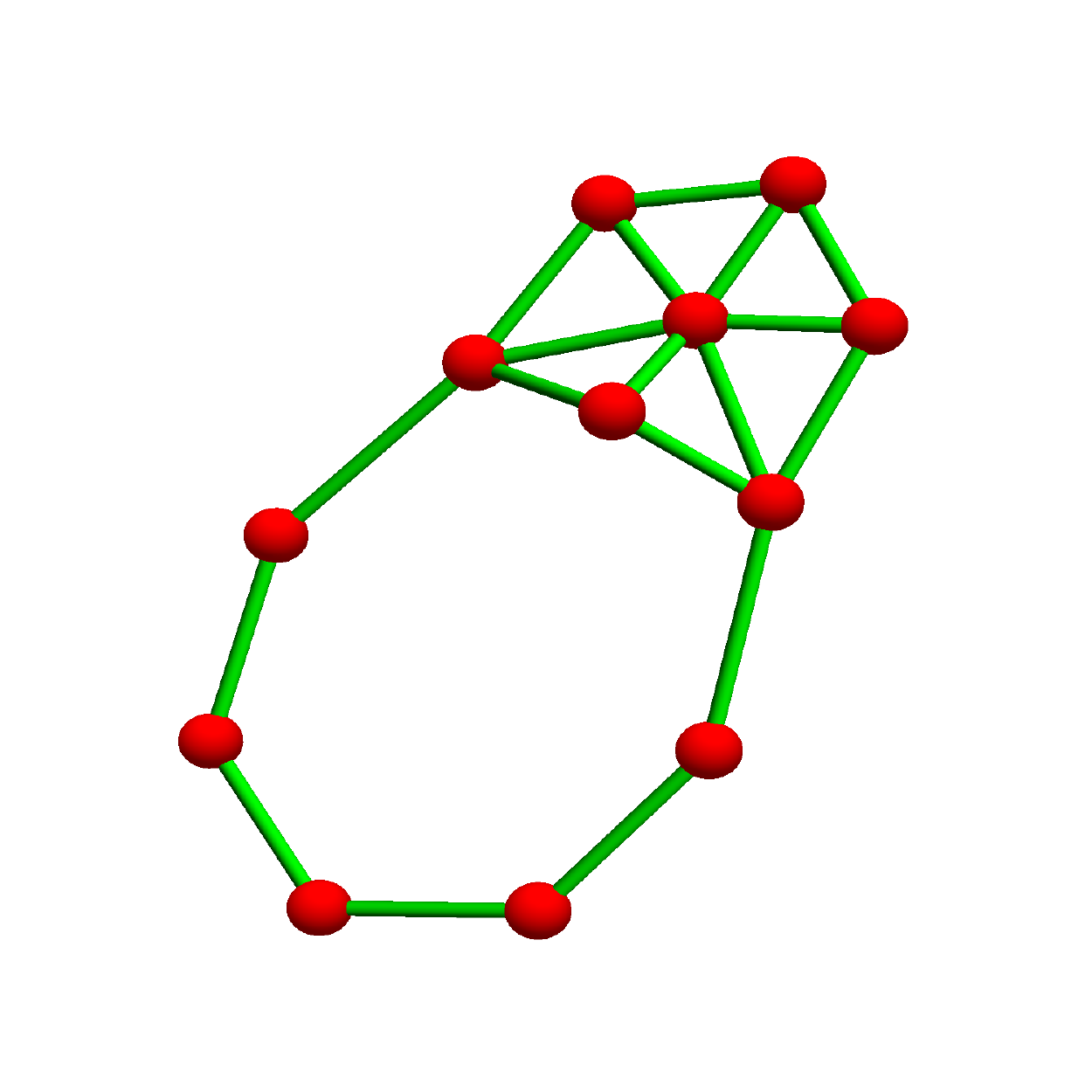}}
\scalebox{0.22}{\includegraphics{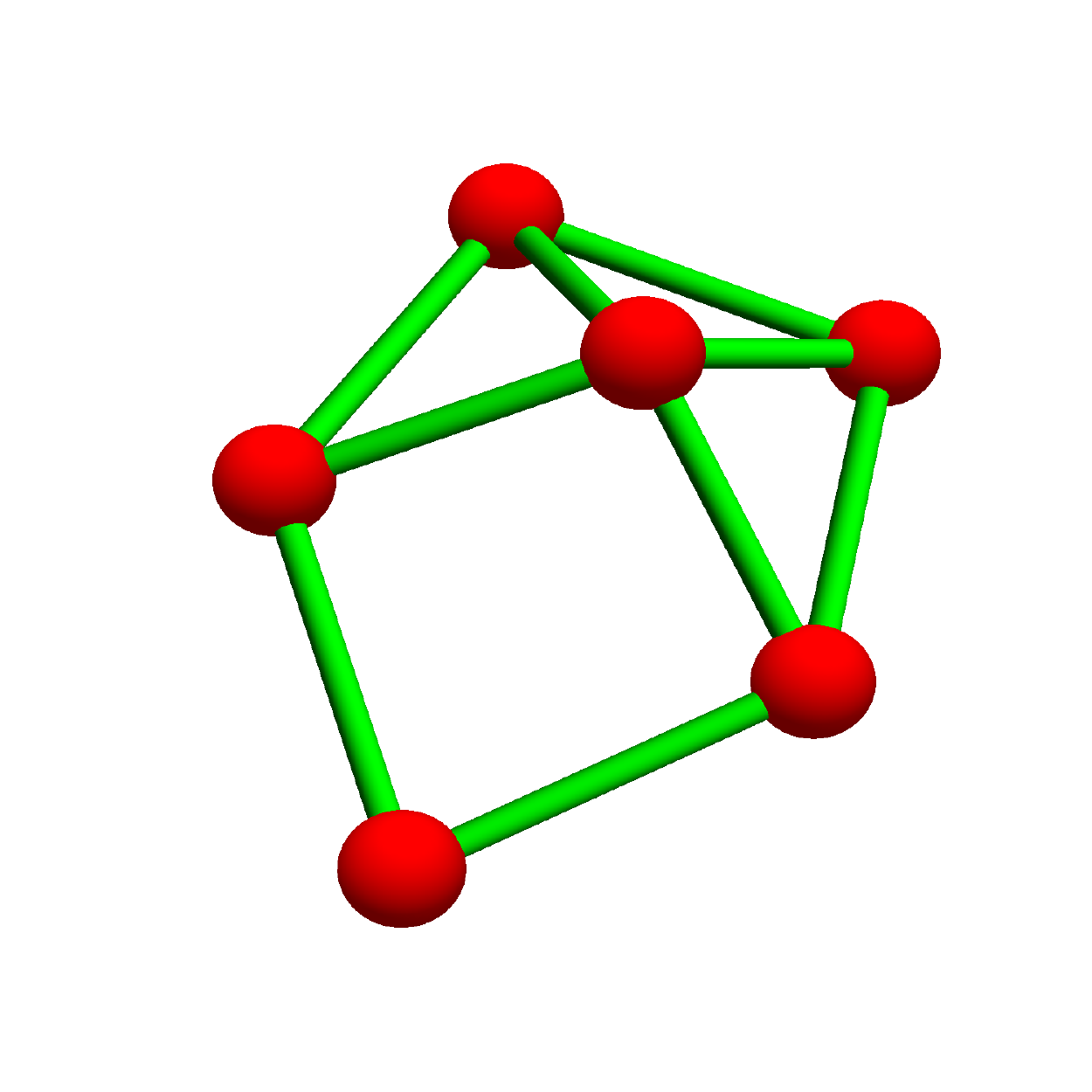}}
\scalebox{0.22}{\includegraphics{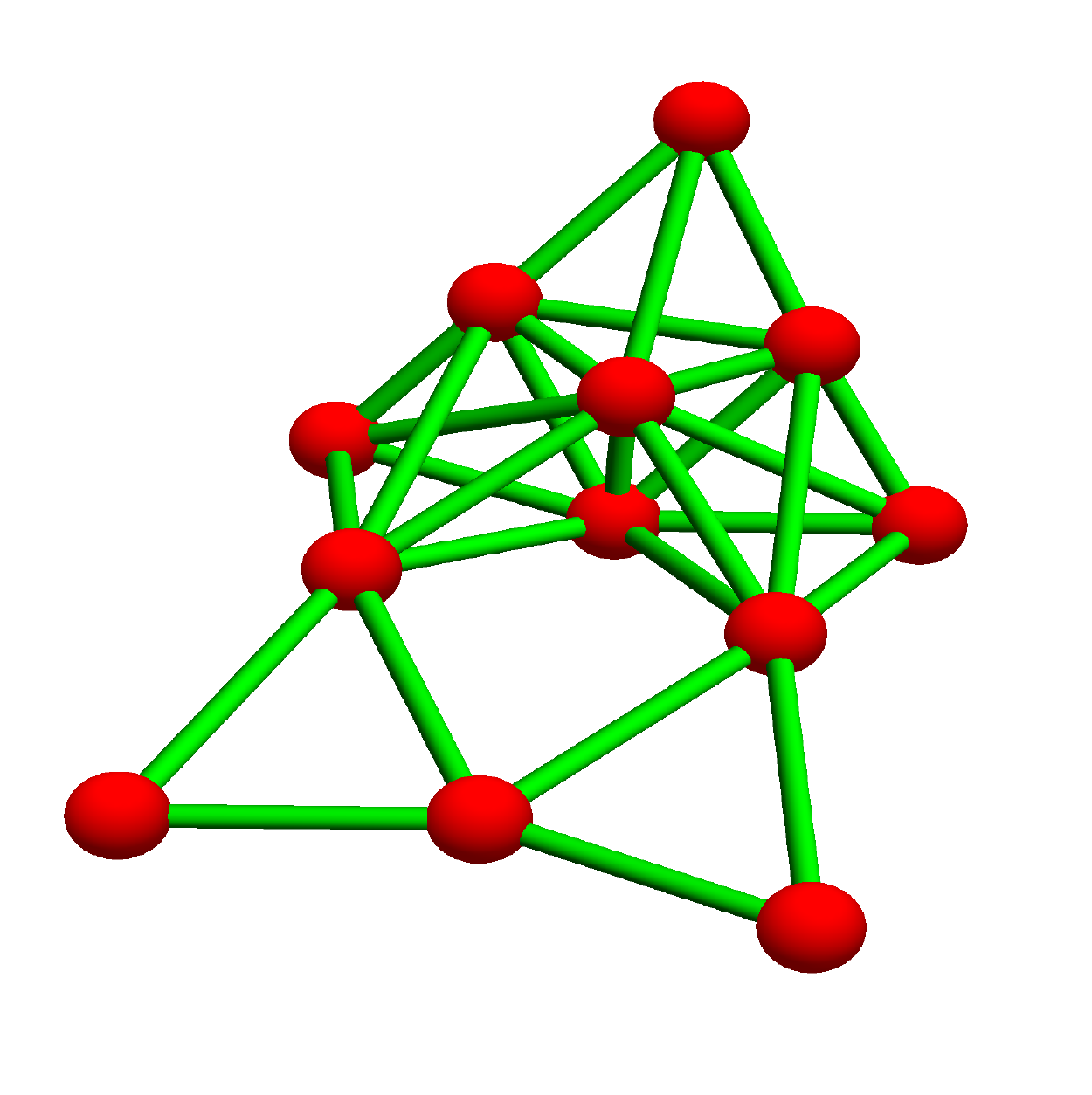}}
\caption{
\label{graphs}
We see the house graph $G$, its Barycentric refinement $G_1$, the line graph $G_L$
and the connection graph $G'$. Both $G_L$ and $G_1$ are subgraphs of $G'$. 
The hydrogen formula relates in full generality for any $G$
the growth of the random walk in the line graph $G_L$
with the growth of the random walk in the connection graph $G'$. 
}
\end{figure} \end{center}

\section{Random walks}

\paragraph{}
Squaring the hydrogen relations $L-L^{-1}=|H|$ gives $L^2-2 + L^{-2} = |H|^2$. We see 
that the sequence of vectors $\psi(n) = L^n \psi$ indexed by $n \in \ZZ$ satisfy 
the in both space and time discrete Laplace recursion relation
$$  \Delta \psi = |H|^2 \psi  \; ,  $$ 
where $\Delta u(n)=u(n+2)-2 u(n) +u(n-2)$ is
a discrete Laplacian. It is really interesting that we get here a {\bf two sided random walk}
$\psi(n)$. This defines a scattering problem as we can look
both at the asymptotic of $\psi(n)$ for $n \to +\infty$
and $n \to -\infty$. The limit $n \to \infty$ is well understood as $L$
is a non-negative matrix which is irreducible. The vectors $\psi(n)/r^n$
converge to the Perron-Frobenius eigenvector, where $r$ is the spectral radius. 
However $\psi(-n)/r^n$ behaves differently as the eigenfunction is different.

\begin{center} \begin{figure}
\scalebox{0.32}{\includegraphics{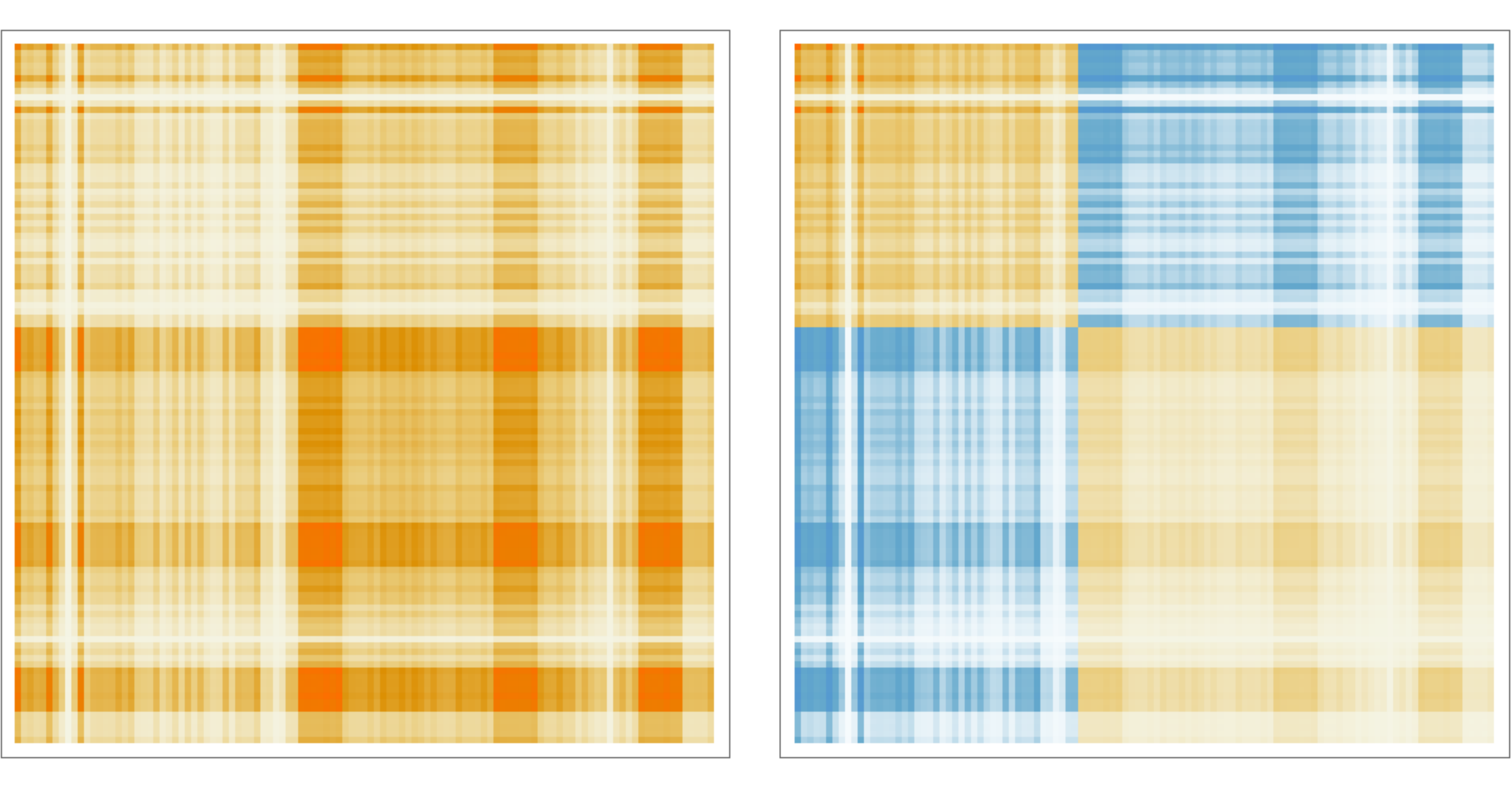}}
\caption{
\label{limiting}
The matrix $L^{2n}/\rho^{2n}$ converges to $P$, the projection of the
Perron-Frobenius operator $v \otimes v$, where $v$ is the eigenvectors
of the largest eigenvalue $\rho(L)$ of $L$. Going backwards, the matrix
$L^{-{2n}}/\rho^{2n}$ however converges onto the projection $w \otimes w$,
where $w$ is the eigenvector of $1/\rho$ of $L$. Unlike the Perron-Frobenius operator
$v$ which is in the positive quadrant, the eigenvector $w$ takes different signs. 
}
\end{figure} \end{center}

\paragraph{}
There is some affinity with a transfer matrix $A$ for 1-dimensional
Schr\"odinger equation $u(n+1) - u(n) + u(n-1) = V(n) u(n)$, where $V(n)$
is a $m \times m$ matrix. 
The model is there also known as a discrete Schr\"odinger operator on the strip
and used to capture features of a two-dimensional discrete operator using methods
from one dimensions. There, the $2m \times 2m$ matrix $A$ is symplectic
(meaning $A^T J A = J$ with the standard symplectic matrix $J$ satisfying $J^2=-1$
and $J^T=-J$). The matrix $L^2$ however is not symplectic in general as its size is 
not necessarily even. While any symplectic matrix $S$ has a block structure 
$$     S      = \left[ \begin{array}{cc} A & B \\ C & D \\ \end{array} \right], 
       S^{-1} = \left[ \begin{array}{cc} D^T & -B^T  \\ -C^T & A^T \end{array} \right] \; , $$
the connection Laplacian $L$ in the circular case satisfies
$$     L      = \left[ \begin{array}{cc} A & B \\ C & D \\ \end{array} \right],
       L^{-1} = \left[ \begin{array}{cc} -D & B \\ C & -A \\ \end{array} \right]    \; . $$
Still, in general, for $1$-dimensional complexes, we have the spectral symmetry 
$\sigma(L^2) = \sigma(L^{-2})$ \cite{DyadicRiemann} 
which is shared by symplectic matrices. This motivates
to investigate whether $L^2$ is conjugated to a symplectic matrix. 

\paragraph{}
A symplectic matrix always has an even size and determinant $1$. It has enjoys 
the spectral symmetry $\lambda \leftrightarrow 1/\lambda$, we know to happen 
for $L^2$. The connection Laplacian $L$ has determinant $-1,1$. If the size of $L$ 
is even, then $L^2$ has a chance to be symplectic because $L^2$ always has determinant 
$1$ and the evenness obstacle is removed. Indeed: 

\begin{propo}[Symplectic Kirby connection]
If $G$ is a graph with $n$ vertices and $m$ edges 
and $n+m$ is even, then $L^2$ is similar to a symplectic matrix. 
\end{propo}
\begin{proof}
The spectral symmetry implies that the characteristic polynomial of $L^2$ is
reciprocal. A theorem of David Kirby (see Theorem A.1 in \cite{Rivin}) implies
that the square $L^2$ of $L$ is similar to a symplectic matrix. 
\end{proof} 

\paragraph{}
The space of solutions of the $1$-dimensional Jacobi equation
$$ \Delta \psi(n)=|H|^2 \psi(n)  $$
with $\Delta u(n)=u(n+2)-2 u(n)+u(n-2)$ is $4n$-dimensional because we
have a second order recursion and two independent lattices. 
Because the sign-less Hodge operator $|H|$ has non-negative spectrum, 
it is not so much a ``discrete wave equation" as a ``discrete Laplace equation". 
It shows the growth for $|n| \to \infty$ as harmonic functions do. 
As the wave equation needs an initial velocity and initial position and 
the Laplacian $\Delta$ has a smallest unit translation $2$, 
it is natural to invoke quaternions: 

\begin{lemma}[Quaternion initial value]
The discrete Laplace equation $\Delta \psi(n,x)=|H|^2 \psi(n,x)$ has a unique solution which is
determined by a  quaternion valued initial field $x \to \psi(0,x) \in \HH$.
\end{lemma}
\begin{proof}
Given a quaternion-valued function $\psi(x)=\psi_0(x) + i \psi_1(x) + j \psi_2(x) + k \psi_3(x)$
on the simplicial complex, we can span the solution space of the (4th order) Jacobi equation
with 
$$L^{2n} \psi_0, L^{-2n} \psi_1, L^{2n+1} \psi_2, L^{-2n-1} \psi_3 \; . $$ 
Both vector spaces are $4n$-dimensional if $n=e+f$ is the number of simplices in the 
$1$-dimensional simplicial complex defined by $G$. 
\end{proof} 

\section{Remarks}

\paragraph{} 
It should be noted that the hydrogen relation $L-L^{-1}=|H|$ holds also
over other fields $F$ rather than the familiar real numbers. 
We can for example look at the relation over the finite field $F=\ZZ_p$ with a prime $p$. 
Because the matrices $L$ and $L^{-1}$ are integer-valued, we just can look at all the numbers 
modulo $p$. Now, the random walk $L^n \psi$ is a cellular automaton in the sense of Hedlund \cite{Hed69}. 
The alphabet is the set of $F$-valued functions from the simplicial complex $G$ to $F$. 
As a random walk, the state $L^n \psi$ is a path integral, summing over all paths of length $n$ 
in the graph $G'$ in which some self loops are attached to each node. 

\begin{propo}[Reversible cellular automaton]
The hydrogen relation $|H|=L-L^{-1}$ still holds over finite fields $F$. The corresponding 
random walk $L^n \psi$ is a reversible cellular automaton over a finite alphabet $F^G$. 
\end{propo}

A cellular automaton defines a homeomorphism $T$ on the compact metric space 
$\Omega=A^Z$, where $A=F^G$ is the finite alphabet of all $F$ valued functions on $G$. 
The dynamics $T$ commutes with the shift. This is the point of view of Hedlund \cite{Hed69}. 
By the theorem of Curtis-Hedlund and Lyndon, the shift commuting property is equivalent to the 
existence of a finite rule involving only a finite neighborhood.  
In the theory of cellular automata, one looks at the attractor $\bigcap T^k \Omega$ 
which defines a subshift of finite type.  
Of interest is the structure of invariant measures or whether the automaton is prime. 
(See i.e. \cite{HofKnill} which sees cellular auatomata in the eyes of Hedlund.)

\paragraph{}
In the Barycentric limit, both the spectral measure of $L$ and the spectral measure of
$H$ converge. In one dimension, the limiting operator of $H$ is given by a Jacobi matrix
while $L$ is a Jacobi matrix on a strip. It still has the property $L-L^{-1}=|H|$. 
Because already after one Barycentric refinement, we have a bipartite
graph, the sign-less Laplacian $|H|$ is similar to $H$. The liming density of states of 
both $L$ and $H$ are known. The limiting spectral function for $H$ is $F(x)=4 \sin^2(\pi x/2)$. 
It satisfies $F(2x)=T(F(x))$ with $T(x)=x(4-x)$. In the finite dimensional case, the spectral 
function is defined as $F_n(x) = \lambda_{[n x]}$, where $[n x]$ is the floor function and
the eigenvalues are $\lambda_0, \dots, \lambda_{n-1}$. Since the eigenvalues $\mu_k$ of $L$ 
satisfy $\phi(\mu_k) = \mu_k - 1/\mu_k = \lambda_k$, the spectral function of $L$ is the 
pull-back under $\phi$. 

\begin{propo}
The hydrogen relation $|H|=L-L^{-1}$ holds also in  the Barycentric limit. The operator $L$
as well as its inverse $L^{-1}$ both remain bounded. 
\end{propo}

We have made use of this already when looking at the limiting Zeta function for $L$
\cite{DyadicRiemann}.

\paragraph{}
For $1$-dimensional Schr\"odinger operators \cite{Pastur,Carmona,Cycon}, 
given in the form of general Jacobi matrices $Hu(n) = a(n) u(n+1) + a(n-1) u(n-1) + b(n) u(n)$,
one has hopping terms $a(n)$ attached to directed edges, and terms $b(n)$ 
attached to self-loops. Adjacency matrices of weighted graphs or
Laplacians (which have row summing up to zero) are also called
``discrete elliptic differential operators" \cite{VerdiereGraphSpectra} and
the zero sum case is a ``harmonic Laplacian". The circular graph case has also a ``covering version"
where one looks at the Floquet theory of periodic operators or more generally at
almost periodic or random cases like ``almost Mathieu"
$a(n)=1,b(n)=c \cos(\theta + n \alpha)$, where 
$a(n)=f(T^nx),b(n)=g(T^nx)$ are defined by translations on a compact
topological group. The flat Laplacian $H u(n)=u(n+1)-2 u(n)+u(n-1)$ on $\ZZ$ 
is a bounded operator on $l^2(\ZZ)$ has the spectrum on $[0,4]$.
We usually don't think about this operator as an almost periodic operator, but its nature is
almost periodic if we look at it as a Barycentric refinement limit, where the hydrogen relation
$|H|=L-L^{-1}$ still works. In order to see the Jacobi structure in the cyclic case, we have
to order the complex as $Z=\{ \dots \{-1,0\},\{0\},\{0,1\}, \{1\}, \{1,2\} \dots \}$
as seen in Figure~(\ref{freelaplace}). The hydrogen relation $|H|=L-L^{-1}$ still holds
and $L$ is a Jacobi matrix "on the strip". With that ordering we have 
$|H|u(n)=u(n+2)+2 u(n)+u(n-2)$, which is isospectral to $H u(n)=-u(n+2)+2 u(n)-u(n-2)$.

\paragraph{}
For any $1$-dimensional complex $G$ with $n$ simplices and any discrete Laplacian $H$ close to 
the standard Hodge Laplacian $|H|$ satisfying $H(x,y)=0$ if $x \cap y = \emptyset$, 
there is a connection operator $L$ such that 
$H=L-L^{-1}$ holds. In that case, both $L$ and $L^{-1}$ have zero entries $L(x,y)=L^{-1}(x,y)$ if $x,y$
do not intersect. We just need to compute the determinant of the derivative $d\phi$ of the map 
$\phi: L \to \pi(L-L^{-1})$ on the finite dimensional space 
$X = \{ A \in M_n(R) \; | \; A=A^T, A(x,y)=0$ if $x \cap y = \emptyset.\}$. Here $\pi$ is the
projection onto $X$. Given an operator $H \in X$. We want to write it as $H=L-L^{-1}$. One can construct $L$
near a known solution by apply the Newton step to the equation $\phi(L)=H-L+L^{-1} =0$ in order 
to solve the $H_{ij} - L_{ij} + g_{ij}(L) = 0$, which are rational functions in the unknown
entries $L_{ij}$ as $g=L^{-1}$ has explicit Cramer type rational expressions.

\paragraph{}
An important open question is how
to extend the hydrogen relation to higher dimensional complexes. We would like to write the sign-less
Hodge operator $|H|$ as a sum $L-L^{-1}$ where $L$ has the spectral symmetry that $L^2$ and $L^{-2}$
have the same spectrum. If such a symmetric extension is not possible,
we can still look at the relation $L-L^{-1} = K$ anyway. It is just that $K$ is not a Hodge
Laplacian any more. In higher dimensions, the spectrum of $H$ can have
negative parts. Still, we can interpret the solution $L^n \psi$ of the reversible random walk 
as a solution of the Laplace equation $\Delta \psi = K^2 \psi$. 

\paragraph{}
Next, we look at the question about the robustness of the hydrogen relation. Can it be extended to 
situations in which the interaction between simplices is not just $L(x,y)=1$ but a number $1+\epsilon(x,y)$? 
The answer is yes, (in some rare cases like the circular a conservation law has to be obeyed as
otherwise, the symmetry produces a zero Jacobean determinant not allowing the implicit function
theorem to be applied. It is quite remarkable that $|H|=L-L^{-1}$ for finite range
Laplacians $H,L$ can be perturbed so that in the perturbation, still both $L$ and $H$ are finite range
implying that $L^{-1}$ has finite range. The simplest case is the Jacobi case, where the graph is a
linear graph or circular graph. The hydrogen relation might go over to perturbations in 
the infinite dimensional case as we have a Banach space of bounded operators. A technical difficulty is
to verify that the Jacobean operator of the map $L \to L-L^{-1}$ is bounded and invertible. 

\paragraph{}
For random Jacobi matrices, where $a(n),b(n)$ are defined by a dynamical system $T$, the relation $L-L^{-1}=H$
requires that $T$ is renormalized in the sense that it is an integral extension of an other system.
(The word ``random operator" is here used in the same way than ``random variable" in probability theory; 
there is no independence nor decorrelation assumed.)
Let us restrict to $\ZZ$, so that time is $1$-dimensional and where the $a(n),b(n)$ can be given
by an integral extension $S$ of an automorphism $T$ of a probability space or an integral extension $S$
of a homeomorphism $T$ of a compact metric space. (An integral extension of $S:X \to X$ is 
$T((x,1)) = (x,2)$ and $T((x,2))=(Sx,1)$. It satisfies $S^2=T$ so that $S^2$ is never ergodic.
Not all dynamical systems are integral extensions; mixing systems are never integral extensions.)

\paragraph{}
We believe that for any random Jacobi matrix $H u(n)=a(n) u(n+2)+a(n-2) u(n-2)+b(n) u(n)$ close enough to $a=0,b=2$,
there is a Jacobi matrix $L u(n) = c(n) u(n+2) + d(n) u(n+1) + c(n-2) u(n-2) + d(n-1) u(n-1) 
+ e(n) u(n)$ of the same type for which the hydrogen relation $H=L-L^{-1}$ holds. 
The condition of being a Jacobi matrix
of this type can be rephrased as an equation $G(L,H)=0$ by bundling all conditions
$[H-L+L^{-1}]_{i,j}=0$ for all $i,j$ and $L_{i,j}=0$ for all $|i-j|>2$. If 
$\partial_L G(L,H)$ is invertible, it is possible to compute the functional derivative of 
$\psi_{ij}(L)=[H-L+L^{-1}]_{ij}$ and assure that the inverse $L^{-1}$ is a Jacobi matrix in a strip.
The hydrogen relation $H=L-L^{-{1}}$ gives then that $L^{-1}$ is a Greens function of the same 
finite range. A priori, it only is a Toeplitz operator and not a Jacobi matrix. It would of course be
nicer to have explicit formulas for $L$, similarly as we can compute $D$ satisfying $L=D^2+E$ for $E$
in the resolvent set of $L$ in terms of Titchmarsh-Weyl $m$-functions. 

\begin{center} \begin{figure}
\scalebox{0.48}{\includegraphics{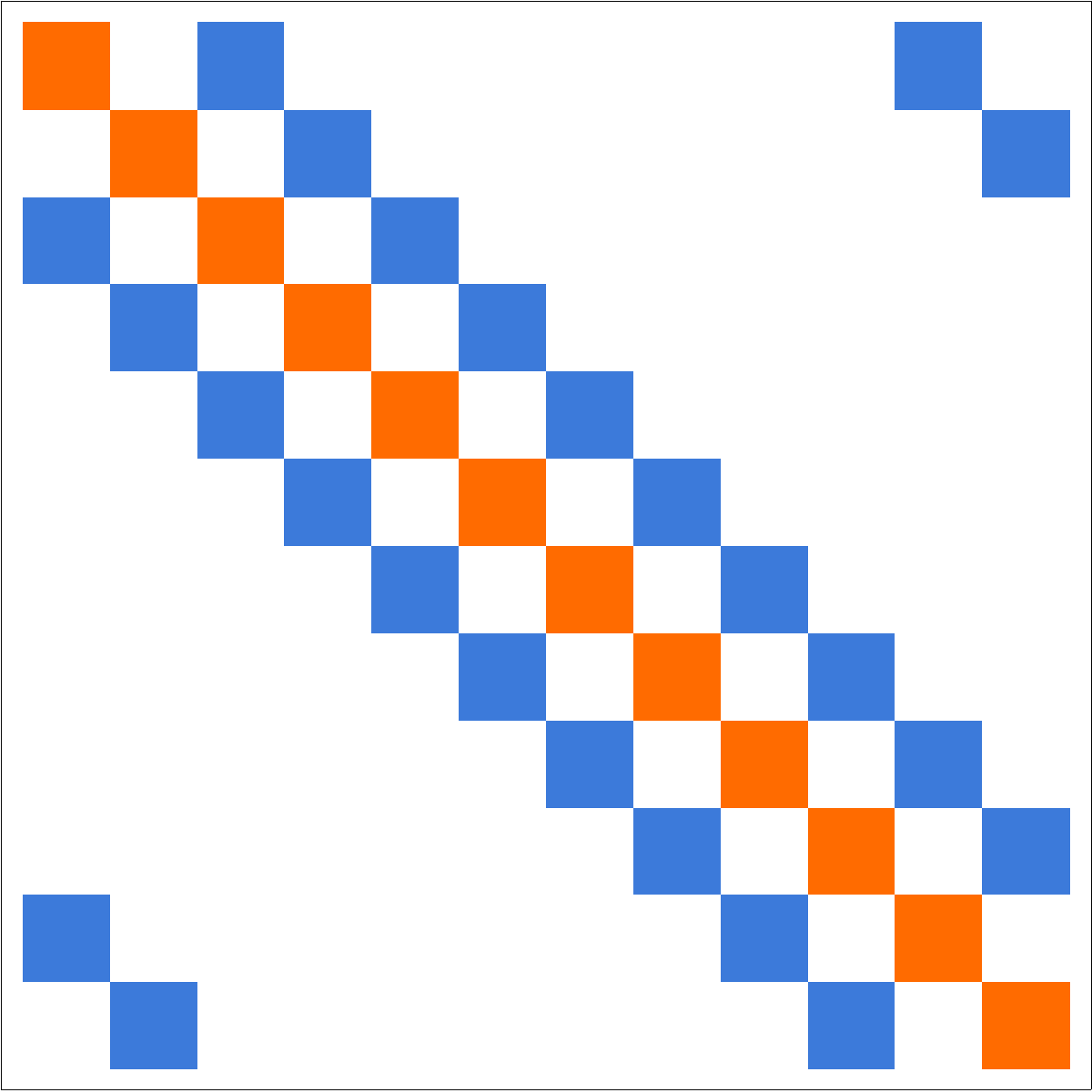}}
\scalebox{0.48}{\includegraphics{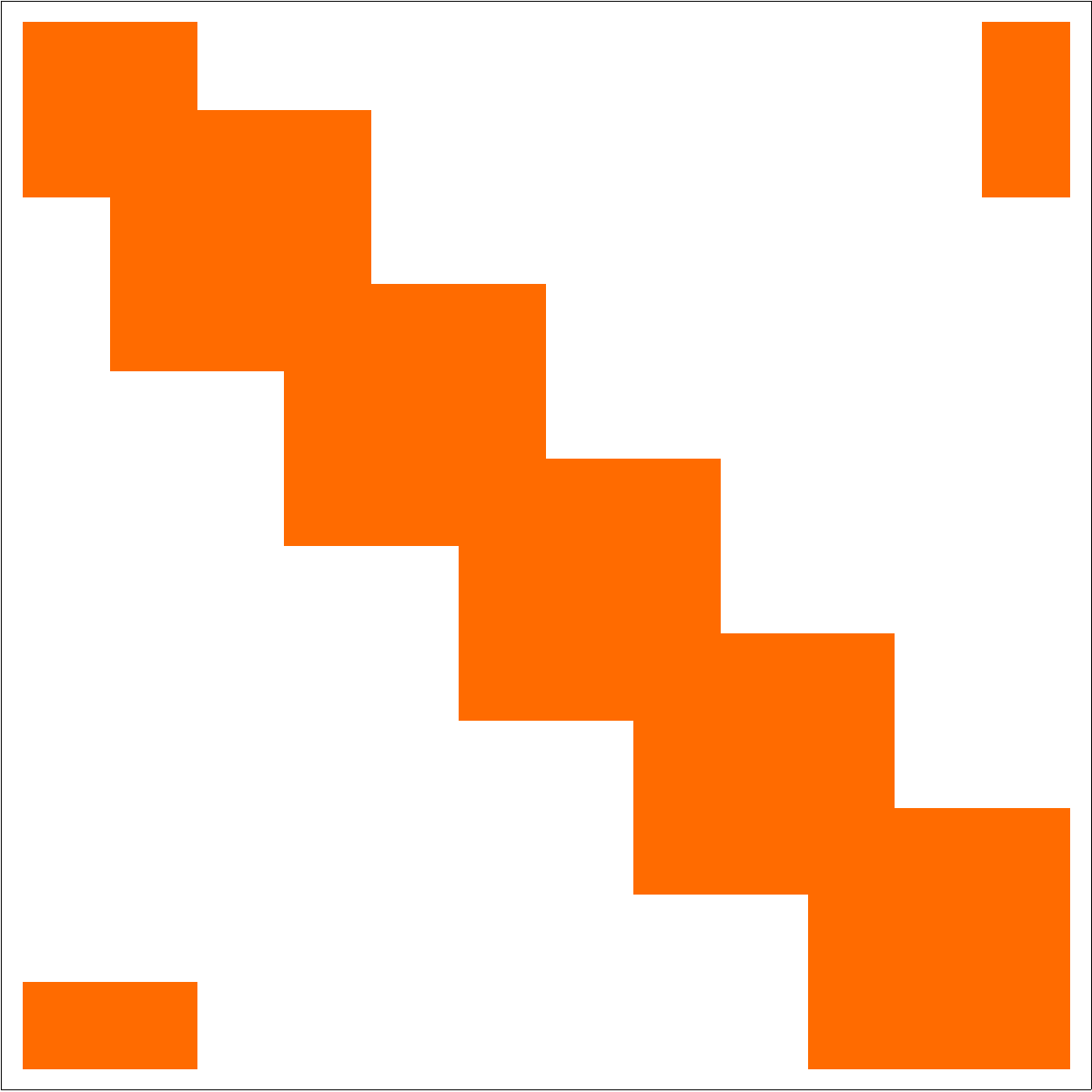}}
\caption{
\label{freelaplace}
The Hodge operator $H$ of the Laplacian $H_0$ on the circular graph is seen
here with the ordering so that $H$ is again a Laplacian, but on a larger scale
(this is the trivial renormalization picture) and the connection 
Laplacian $L$. We display here periodic matrices but in the limit we have operators on $l^2(\ZZ)$ 
the hydrogen relation $|H|=L-L^{-1}$ still holds. The matrix $L$ can not be diagonalized with
Fourier. We needed the hydrogen relation to diagonalize $L$ \cite{DyadicRiemann}.
}
\end{figure} \end{center}

\paragraph{}
Simplicial complexes define a ring. The addition in this "strong ring" is obtained by taking the disjoint union. 
This monoid can be extended with the Grothendieck construction (in the same way as integers are formed from
the additive monoid of natural numbers or fractions are formed from the multiplicative monoid of non-zero integers) 
to a group in which the empty complex $0$ is the zero element. The Cartesian product of
two complexes is not a simplicial complex any more. One can however look at the ring generated
by simplicial complexes. This ring is a unique factorization domain and the simplicial complexes
are the primes \cite{StrongRing}. We call it the strong ring because the corresponding connection
graphs multiply with the strong product in graph theory. 
One can now define both a connection matrix $L(A \times B)$
as well as a Hodge operator $H(A \times B)$. If $A$ has $n$ elements and $B$ has $m$ elements,
then $L(A \times B)$ and $H(A \times B)$ are both $nm \times nm$ matrices. If $\lambda_k$ are
the eigenvalues of $L(A)$ and $\mu_k$ are the eigenvalues of $L(B)$, then $\lambda_k \mu_l$ are
the eigenvalues of $L(A \times B)$.  If $\lambda_k$ are the eigenvalues of $H(A)$ and 
$\mu_k$ are the eigenvalues of $H(B)$ then $\lambda_k + \lambda_l$ are the eigenvalues of 
$H(A \times B)$. We can now ask whether the product operators $L(A \times B)$ and 
$H(A \times B)$ still satisfy the hydrogen relations. This is not true, but also not to
be expected. What sill happens in general for all simplicial complexes that 
$g=L(A \times B)^{-1}$ satisfies the energy theorem
$\sum_{x \in A \times B} \sum_{y \in A \times B} g(x,y) = \chi(A \times B) = \chi(A) \chi(B)$. 
Also, in the $1$-dimensional case, still $L^2(A \times B)$ has the same 
spectrum than $L^{-2}(A \times B)$. 

\paragraph{}
If we think of $L-L^{-1}$ as a $1$-dimensional derivative and $L^2-2 + L^{-2}$ as a second
derivative, then we can look at $\Delta = L_A^2-2 + L_A^{-2} + L_B^2-2 + L_B^{-2}$ as a discrete
Laplacian leading to a random walk $\Delta u = L^2(A \times B) u$ parametrized by a two 
dimensional time. Again, given a quaternion valued function on the simplices of the graph gives
now a unique solution $u(n,m) = L_A^n L_B^m u$. 
This can be generalized to any dimension. In this model, we should think of the graph $G$ as
the ``fibre" over a discrete lattice. An other way to think about it is to see the lattice
$Z^2$ as a two-dimensional time for a stochastic process on the graph. The stochastic process is
interesting because it describes a classical random walk in a graph if time is positive.
That this random walk can be reversed allows to walk backwards, again with finite propagation speed. 
While the backwards process given by $L_A^{-1}$ or $L_B^{-1}$ is still propagating in the same 
way than the forward random walk, it is not Perron-Frobenius. There is an arrow of time. 

\paragraph{}
What is interesting about the hydrogen formula $L-L^{-1} = |H|$ is that the scales are fixed. 
If we insist $L$ to be integer valued, we can not wiggle any coefficients, nor multiply $H$ with 
a factor. In comparison, in \cite{Kni98} we have for numerical purposes time discretized a 
Schr\"odinger equation $i h u' = L u$ with bounded Hamiltonian 
$L$ by scaling the constant such that $L$ has a small norm. We then looked at a 
discrete time dynamics for a deformed operator which interpolates the continuum differential equation.
While the deformation does change the spectrum it does not the nature
of the spectral measures (like for example whether it is singular continuous, 
absolutely continuous or pure point), the method is especially useful, using the Wiener 
criterion, to measure numerically whether some discrete spectrum is present. 
The analogy is close: we had defined $U=L \pm i \sqrt{1-L^2}$ 
so that $U+1/U = 2 L$. If $\psi(t) = \exp(it L/h)$ is the quantum evolution, then
$\psi(n+1)+\psi(n-1) = 2 L \psi(n)$ with $\psi(n) = U^n \psi$.
Now $U^n = \exp(i n L) = \cos(i n \arccos(L))  + i \sin(i n \arccos(L)) = T_n(L) + i R_n(L)$
which are Tchebychev polynomials of $L$. This can be done faster than matrix exponentiation
and leads to finite propagation speed, a property which $\exp(i t L/h)$ does not enjoy. For
numerical purposes, it allowed to make sure that no boundary effects play a role or that the
process can be run using exact rational expressions. 
We have used the method to measure the spectrum of almost periodic operators like
the almost Mathieu operator, where the spectrum is understood or magnetic two-dimensional
operators, where the spectrum is not understood.  In the present hydrogen relation, we can
not scale $L$. In some sense, the Planck constant is fixed. 

\paragraph{}
It could be interesting to look at the frame work in the context of a random geometric model called
``Causal Dynamical Triangulations" \cite{CausalDynamicalTriangulation}.
This is a situation where a sequence $G_k$ of 
$1$-dimensional simplicial complexes is given by some time evolution. For the following it is not relevant
how $G_k$ are chosen as long as they are given by a dynamical process and the
number of simplices is fixed. If $L_k$ are the connection Laplacians of $G_k$, we can look at the system 
$L_{k+1} - L_{k-1}^{-1} = V_k$. 
The random walk becomes a cocycle $\psi_k = L_k \cdots L_2 L_1 \psi$ and the spectral radius $\rho$ becomes a 
Lyapunov exponent. 

\paragraph{}
If the sequence of graphs with the same number of vertices and edges, the orbit closure
defines an ergodic random process. By Oseledec's theorem, the Lyapunov exponent exists. Mathematically, the problem 
$u(n+2) - 2 u(n) + u(n-2) = V_n u(n)$
is now a {\rm ergodic Jacobi operator} on the strip, in which the graphs $G_k$ are the fibers. 
The condition that the total number of vertices and edges remain constant
is a topological condition. We could also insist that $G_k$ have all the same topology. 
The solution space is then described by a quaternion-valued field:
the hydrogen relations are already second order
and we have an even and odd branch of time $\ZZ$ which evolve independently. So there are four wave function
components needed to determine the initial condition. 

\paragraph{}
The ergodic situation is interesting because in a random setup, there is a chance to get 
localization and so obtain solutions $\psi(n) = L^n \psi$ which go to zero both for $n \to \infty$
as well as $n \to -\infty$. A concrete question is to construct a finite set of graphs
$\{ G_1, \dots G_k \}$ with corresponding connection Laplacians 
$\{ L_1, \dots , L_k \}$ so that for almost all two sided sequences $\omega \in \{ 1,2, \dots, k\}^{\ZZ}$ 
there exists $\psi_{\omega}(0)$ such that the random walk in a random environment
\begin{equation}
\label{randomwalkrandomenvironment}
\psi_{\omega}(n) = L_{\omega(n)} L_{\omega(n-1)} \dots, L_{\omega(2)} L_{\omega(1)} \psi(0)
\end{equation}
has the property that
$|\psi_{\omega}(n)| \to 0$ for $|n| \to \infty$. This does not look strange if comparing with
the Anderson localization picture, where symplectic transfer matrices lead by Oseledec to 
almost certain exponential growth in one direction, but where it is possible that for almost all
$\omega$ in the probability space, one has a complete set of eigenfunctions (which decay both 
forward and backward in time). In our case now, it is all just random walks $\psi(n)$ 
on a dynamically changing graphs. 

\paragraph{}
The dynamical triangulation picture suggests to start with a single abstract simplicial 
complex $G$ and think of $G$ as a ``space time manifold" with the property $G$
is a disjoint union of finite abstract simplicial complexes $G=\{ G_k \}_{k \in \ZZ}$ 
for which each $G_k$ has the same cardinality $m$. We don't really need to assume the ``time slices" $G_k$ 
to be $1$-dimensional. Each $G_k$ has a connection Laplacian $L_k$. Because there are
only a finite set $A$ of simplicial complexes with $m$ simplicies, the space-time complex 
$G$ defines an element in the compact metric space $A^{\ZZ}$. As usual in symbolic
dynamics, the single sequence $G$ defines its orbit closure $\Omega$
(the subset of all accumulation points which is a shift-invariant and closed subset).
A natural assumption (avoiding many-world interpretations),
is that the shift is uniquely ergodic; the complex $G$ alone determines
a unique natural probability space $(\Omega,\mu)$.

\paragraph{}
One can now study the properties of the random walk $\psi_{\omega}(n)$ given in (\ref{randomwalkrandomenvironment}).
The unimodularity theorem assures that $L_k$ are invertible and in the $1$-dimensional
case (if $m$ is even) also symplectic. The structure of the inverse
$g=L^{-1}$ assures that the random walk is two sided and that also the inverse is a random walk as
the transition steps $g(x,y)=L^{-1}(x,y)$ are defined by the Green star formula 
\cite{ListeningCohomology}. The Anderson localization picture suggests the existence
of examples, where $||\psi_{\omega}(n)||_2$ is bounded. The $\psi(n)$ are solutions to a reversible
random walk, but they are also localized. They solve $\Delta \psi = K \psi$
with $K=(L-L^{-1})^2$ which is a Wheeler-DeWitt type eigenvalue equation.
It describes the reversible random walk in a random environment so that 
solutions are path integrals. The model is robust in the sense that if
we perturb $K=H^2$, we can still write $L-L^{-1}=H$ with operators $L,L^{-1}$ which have
the same support and so finite propagation speed.

\vfill 
\pagebreak

\section{Mathematica Code} 

\paragraph{}
We construct a random graph, compute both $L$ and $|H$
then check the hydrogen relation $L-L^{-1}-|H|=0$. 

\begin{tiny}
\lstset{language=Mathematica} \lstset{frameround=fttt}
\begin{lstlisting}[frame=single]
(* The sign-less Hydrogen relation, O.Knill, 2/10/2018  *)
{v,e}={30,80}; s=RandomGraph[{v,e}]; n=v+e;
bracket[x_]:={x}; set[x_]:={x[[1]],x[[2]]};        
G=Union[Map[set,EdgeList[s]],Map[bracket,VertexList[s]]];     
m[a_,b_]:=If[SubsetQ[a,b]&&(Length[a]==Length[b]+1),1,0];
d=Table[m[G[[i]],G[[j]]],{i,n},{j,n}]; (* signl. deriv *)
Dirac=d+Transpose[d];  H=Dirac.Dirac;  (* signl. Hodge *)
L=Table[If[DisjointQ[G[[k]],G[[l]]],0,1],{k,n},{l,n}];
Total[Flatten[Abs[(L-Inverse[L]) - H]]]
\end{lstlisting}
\end{tiny}

The various spectral radius estimates can be compared.
In the various tables, we first compute $\rho=\rho(H)$, then $|\rho|=\rho(|H|)$,
then $r-1/r$ with $r=1+ ({\rm max}_x P(3,x))^{1/3}$,
then $r-1/r$ with $r=1+{\rm max}_{x,y \in E} d(x)+d(y)$,
then the edge estimates and vertex degree estimates which 
involve the diameter of the graph. 

\begin{tiny}
\lstset{language=Mathematica} \lstset{frameround=fttt}
\begin{lstlisting}[frame=single]
(* Estimating the spectral radius,      O.Knill, 2/18/2018  *)
ClearAll["Global`*"];
L[s_]:=Module[{v=VertexList[s],e=EdgeList[s],n,G,L},b[x_]:={x};
  p[x_]:={x[[1]],x[[2]]};G=Union[Map[p,e],Map[b,v]];n=Length[G]; 
  Table[If[DisjointQ[G[[k]],G[[l]]],0,1],{k,n},{l,n}]];
Barycent[s_]:=Module[{v=VertexList[s],e=EdgeList[s],V,W},
  p[x_]:={x[[1]],x[[2]]};  W=Flatten[Table[
  {p[e[[k]]]->e[[k,1]],p[e[[k]]]->e[[k,2]]},{k,Length[e]}]];
  V=Union[v,Map[p,e]];UndirectedGraph[Graph[V,W]]];BC=Barycent;
sprime[s_]:=AdjacencyGraph[L[s]-IdentityMatrix[Length[L[s]]]];
Rho[s_] :=Max[Eigenvalues[1.0     Normal[KirchhoffMatrix[s]]]];
Rhoa[s_]:=Max[Eigenvalues[1.0 Abs[Normal[KirchhoffMatrix[s]]]]];
Walk[s_,k_] :=Max[Map[Total,MatrixPower[AdjacencyMatrix[s],k]]]; 
VD=VertexDegree; (* next comes Shi and Li-Shiu-Chang estimates*)
ShLiShCh[s_]:=Module[{R=GraphDiameter[s],g=VD[s],d,v},
  d=Max[g]; v=Length[VertexList[s]];  N[2d-1/(v (2 R+1))]]; 
StBrHo[s_]:=Module[{e,sp=sprime[s]},(*Stanley-Brualdi-Hoffmann*)
  e=Length[EdgeList[sp]]; u=1+N[(Sqrt[1+8 e]-1)/2]; u-1/u];
Deg2[s_]:=Module[{e=EdgeList[s],v=VertexList[s],r},
  d[k_]:=1+VD[s,v[[e[[k,1]]]]]+VD[s,v[[e[[k,2]]]]];
  r=Max[Table[d[k],{k,Length[e]}]]; N[r-1/r]];
Walk3[s_]:=Module[{r=1+(Walk[sprime[s],3])^(1/3)},N[r-1/r]]; 
F[s_]:={Rho[s],Rhoa[s],Walk3[s],Deg2[s],StBrHo[s],ShLiShCh[s]};
Table[F[CompleteGraph[{3,k}]],{k,3,9}] // MatrixForm
Table[F[CycleGraph[k]],{k,4,10}]       // MatrixForm
Table[F[CompleteGraph[k]],{k,2,8}]     // MatrixForm
Table[F[StarGraph[k]],{k,4,10}]        // MatrixForm
Table[F[WheelGraph[k]],{k,5,11}]       // MatrixForm
Table[F[PetersenGraph[6,k]],{k,2,8}]   // MatrixForm
Table[F[GridGraph[{k,1}]],{k,2,8}]     // MatrixForm
Table[F[GridGraph[{6,k}]],{k,2,8}]     // MatrixForm
Table[F[RandomGraph[{20,5}]],{7}]      // MatrixForm
Table[F[RandomGraph[{20,50}]],{7}]     // MatrixForm
Table[F[RandomGraph[{30,100}]],{7}]    // MatrixForm
Table[F[BC[RandomGraph[{20,100}]]],{7}]// MatrixForm
\end{lstlisting}
\end{tiny}

\vfill
\pagebreak

\section{Measurements}

\paragraph{}
The following tables were obtained by running the code displayed in the above section. In 
each case, we see the spectral radius $\rho$, the sign-less spectral radius $|\rho|$, 
the dual vertex estimate, a three step walk estimate, then a global edge estimate and
finally an upper bound obtained in \cite{LiShiuChan2010}.

\paragraph{}
{\bf Complete bipartite Graphs} are computed from $K_{3,3}$ (utility graph) to $K_{3,9}$. These
are regular graphs for which \cite{LiShiuChan2010} do not apply. Indeed, we see that the
estimate is then below the actual spectral radius $\rho$. 
 
\begin{tabular}{|c|c|c|c|c|c|} \hline
     $\rho$ & $|\rho|$ & Thm~(\ref{dualvertex}) & 3Walk & \cite{BrualdiHoffmann,Stanley1987} & \cite{LiShiuChan2010} \\ \hline
     6. & 6. & 6.85714 & 6.22655 & 8.88889 & 5.96667 \\ \hline
     7. & 7. & 7.875 & 7.23871 & 10.8126 & 7.97143 \\ \hline
     8. & 8. & 8.88889 & 8.24856 & 12.6793 & 9.975 \\ \hline
     9. & 9. & 9.9 & 9.25665 & 14.5115 & 11.9778 \\ \hline
     10. & 10. & 10.9091 & 10.2634 & 16.3213 & 13.98 \\ \hline
     11. & 11. & 11.9167 & 11.269 & 18.1156 & 15.9818 \\ \hline
     12. & 12. & 12.9231 & 12.2739 & 19.8985 & 17.9833 \\ \hline
    \end{tabular}

\paragraph{}
{\bf Cyclic graphs} are computed from $C_4$ to $C_{10}$. Also these are regular graphs.

 \begin{tabular}{|c|c|c|c|c|c|} \hline
     $\rho$ & $|\rho|$ & Thm~(\ref{dualvertex}) & 3Walk & \cite{BrualdiHoffmann,Stanley1987} & \cite{LiShiuChan2010} \\ \hline
     4. & 4. & 4.8 & 4.19371 & 5.24008 & 3.95 \\ \hline
     3.61803 & 4. & 4.8 & 4.19371 & 5.83333 & 3.96 \\ \hline
     4. & 4. & 4.8 & 4.19371 & 6.36744 & 3.97619 \\ \hline
     3.80194 & 4. & 4.8 & 4.19371 & 6.85714 & 3.97959 \\ \hline
     4. & 4. & 4.8 & 4.19371 & 7.31193 & 3.98611 \\ \hline
     3.87939 & 4. & 4.8 & 4.19371 & 7.73832 & 3.98765 \\ \hline
     4. & 4. & 4.8 & 4.19371 & 8.14105 & 3.99091 \\ \hline
    \end{tabular}

\paragraph{}
{\bf Complete Graphs} are computed from $K_2$ (interval) to $K_8$ (an 7 dimensional simplex). These
are regular graphs for which \cite{LiShiuChan2010} do not apply.

 \begin{tabular}{|c|c|c|c|c|c|} \hline
     $\rho$ & $|\rho|$ & Thm~(\ref{dualvertex}) & 3Walk & \cite{BrualdiHoffmann,Stanley1987} & \cite{LiShiuChan2010} \\ \hline
     2. & 2. & 2.66667 & 2.20091 & 2.17116 & 1.83333 \\ \hline
     3. & 4. & 4.8 & 4.19371 & 4.56245 & 3.88889 \\ \hline
     4. & 6. & 6.85714 & 6.22655 & 7.31193 & 5.91667 \\ \hline
     5. & 8. & 8.88889 & 8.24856 & 10.4174 & 7.93333 \\ \hline
     6. & 10. & 10.9091 & 10.2634 & 13.8539 & 9.94444 \\ \hline
     7. & 12. & 12.9231 & 12.2739 & 17.5971 & 11.9524 \\ \hline
     8. & 14. & 14.9333 & 14.2817 & 21.6258 & 13.9583 \\ \hline
    \end{tabular}

\paragraph{}
{\bf Star Graphs} are computed starting with central degree $3$ to central degree $9$.

 \begin{tabular}{|c|c|c|c|c|c|} \hline
     $\rho$ & $|\rho|$ & Thm~(\ref{dualvertex}) & 3Walk & \cite{BrualdiHoffmann,Stanley1987} & \cite{LiShiuChan2010} \\ \hline
     4. & 4. & 4.8 & 4.19371 & 4.56245 & 5.95 \\ \hline
     5. & 5. & 5.83333 & 5.21154 & 5.64311 & 7.96 \\ \hline
     6. & 6. & 6.85714 & 6.22655 & 6.69818 & 9.96667 \\ \hline
     7. & 7. & 7.875 & 7.23871 & 7.73832 & 11.9714 \\ \hline
     8. & 8. & 8.88889 & 8.24856 & 8.76893 & 13.975 \\ \hline
     9. & 9. & 9.9 & 9.25665 & 9.79308 & 15.9778 \\ \hline
     10. & 10. & 10.9091 & 10.2634 & 10.8126 & 17.98 \\ \hline
    \end{tabular}

\paragraph{}
{\bf Wheel Graphs} are computed with central degree 4 up to central degree 10

 \begin{tabular}{|c|c|c|c|c|c|} \hline
     $\rho$ & $|\rho|$ & Thm~(\ref{dualvertex}) & 3Walk & \cite{BrualdiHoffmann,Stanley1987} & \cite{LiShiuChan2010} \\ \hline
     5. & 6.56155 & 7.875 & 6.99565 & 8.64722 & 7.96 \\ \hline
     6. & 7.23607 & 8.88889 & 7.8263 & 9.9 & 9.96667 \\ \hline
     7. & 8. & 9.9 & 8.69993 & 11.0994 & 11.9714 \\ \hline
     8. & 8.82843 & 10.9091 & 9.60356 & 12.2617 & 13.975 \\ \hline
     9. & 9.70156 & 11.9167 & 10.5285 & 13.3969 & 15.9778 \\ \hline
     10. & 10.6056 & 12.9231 & 11.4687 & 14.5115 & 17.98 \\ \hline
     11. & 11.5311 & 13.9286 & 12.4203 & 15.6102 & 19.9818 \\ \hline
    \end{tabular}

\paragraph{}
{\bf Peterson Graphs} ${\rm Peterson}(6,k)$ are computed for $k=2$ to $k=8$. 

 \begin{tabular}{|c|c|c|c|c|c|} \hline
     $\rho$ & $|\rho|$ & Thm~(\ref{dualvertex}) & 3Walk & \cite{BrualdiHoffmann,Stanley1987} & \cite{LiShiuChan2010} \\ \hline
    5.23607 & 6. & 6.85714 & 6.22655 & 12.4305 & 5.99074 \\ \hline
    5.41421 & 5.41421 & 6.85714 & 5.92748 & 10.8126 & 5.99074 \\ \hline
    5.23607 & 6. & 6.85714 & 6.22655 & 12.4305 & 5.99074 \\ \hline
    6. & 6. & 6.85714 & 6.22655 & 12.4305 & 5.99074 \\ \hline
    5.23607 & 5.23607 & 6.85714 & 5.87411 & 10.8126 & 5.99242 \\ \hline
    6. & 6. & 6.85714 & 6.22655 & 12.4305 & 5.99074 \\ \hline
    5.23607 & 6. & 6.85714 & 6.22655 & 12.4305 & 5.99074 \\ \hline
   \end{tabular}

\paragraph{}
{\bf Linear Graphs} are taken from length $1$ to $7$ 

 \begin{tabular}{|c|c|c|c|c|c|} \hline
     $\rho$ & $|\rho|$ & Thm~(\ref{dualvertex}) & 3Walk & \cite{BrualdiHoffmann,Stanley1987} & \cite{LiShiuChan2010} \\ \hline
     2. & 2. & 2.66667 & 2.20091 & 2.17116 & 1.83333 \\ \hline
     3. & 3. & 3.75 & 3.17771 & 3.43141 & 3.93333 \\ \hline
     3.41421 & 3.41421 & 4.8 & 3.78886 & 4.31043 & 3.96429 \\ \hline
     3.61803 & 3.61803 & 4.8 & 3.96987 & 5.02531 & 3.97778 \\ \hline
     3.73205 & 3.73205 & 4.8 & 4.13272 & 5.64311 & 3.98485 \\ \hline
     3.80194 & 3.80194 & 4.8 & 4.16348 & 6.19493 & 3.98901 \\ \hline
     3.84776 & 3.84776 & 4.8 & 4.19371 & 6.69818 & 3.99167 \\ \hline
    \end{tabular}

\paragraph{}
{\bf Grid graphs} $G(6,k)$ with $k=2$ to $k=8$ lead to

 \begin{tabular}{|c|c|c|c|c|c|} \hline
     $\rho$ & $|\rho|$ & Thm~(\ref{dualvertex}) & 3Walk & \cite{BrualdiHoffmann,Stanley1987} & \cite{LiShiuChan2010} \\ \hline
    5.73205 & 5.73205 & 6.85714 & 6.20288 & 11.3786 & 5.99359 \\ \hline
    6.73205 & 6.73205 & 8.88889 & 7.78401 & 15.4104 & 7.9963 \\ \hline
    7.14626 & 7.14626 & 8.88889 & 8.00389 & 18.564 & 7.99755 \\ \hline
    7.35008 & 7.35008 & 8.88889 & 8.211 & 21.2437 & 7.99825 \\ \hline
    7.4641 & 7.4641 & 8.88889 & 8.22357 & 23.6148 & 7.99868 \\ \hline
    7.53399 & 7.53399 & 8.88889 & 8.23608 & 25.7644 & 7.99896 \\ \hline
    7.57981 & 7.57981 & 8.88889 & 8.23608 & 27.7449 & 7.99917 \\ \hline
   \end{tabular}

\paragraph{}
Now, we take {\bf random graphs} with 20 vertices and 4 edges:

 \begin{tabular}{|c|c|c|c|c|c|} \hline
     $\rho$ & $|\rho|$ & Thm~(\ref{dualvertex}) & 3Walk & \cite{BrualdiHoffmann,Stanley1987} & \cite{LiShiuChan2010} \\ \hline
     5.08613 & 5.08613 & 6.85714 & 5.49603 & 6.19493 & 8. \\ \hline
     3. & 3. & 3.75 & 3.17771 & 5.24008 & 4. \\ \hline
     3. & 3. & 3.75 & 3.17771 & 5.02531 & 4. \\ \hline
     3. & 3. & 3.75 & 3.17771 & 5.02531 & 4. \\ \hline
     4. & 4. & 4.8 & 4.19371 & 5.44565 & 6. \\ \hline
     4. & 4. & 4.8 & 4.19371 & 5.44565 & 6. \\ \hline
     2. & 2. & 2.66667 & 2.20091 & 4.8 & 2. \\ \hline
    \end{tabular}

\paragraph{}
{\bf Random graphs} with 20 vertices and 50 edges:

 \begin{tabular}{|c|c|c|c|c|c|} \hline
     $\rho$ & $|\rho|$ & Thm~(\ref{dualvertex}) & 3Walk & \cite{BrualdiHoffmann,Stanley1987} & \cite{LiShiuChan2010} \\ \hline
    10.5926 & 12.7038 & 17.9444 & 15.0334 & 26.5056 & 17.9944 \\ \hline
    9.28796 & 11.0103 & 13.9286 & 12.4608 & 25.4858 & 13.9929 \\ \hline
    11.3603 & 12.7573 & 17.9444 & 15.0837 & 26.7353 & 17.9944 \\ \hline
    10.8116 & 12.1925 & 17.9444 & 14.6776 & 26.1572 & 17.9929 \\ \hline
    10.6623 & 12.5725 & 17.9444 & 14.9725 & 26.2738 & 17.9944 \\ \hline
    10.447 & 12.8607 & 16.9412 & 14.9606 & 26.4671 & 18. \\ \hline
    10.1163 & 11.308 & 15.9375 & 13.2007 & 25.6454 & 15.9929 \\ \hline
   \end{tabular}

\paragraph{}
{\bf Random graphs} with 30 vertices and 100 edges:

 \begin{tabular}{|c|c|c|c|c|c|} \hline
     $\rho$ & $|\rho|$ & Thm~(\ref{dualvertex}) & 3Walk & \cite{BrualdiHoffmann,Stanley1987} & \cite{LiShiuChan2010} \\ \hline
    14.4959 & 16.4624 & 23.9583 & 19.885 & 41.32 & 25.9963 \\ \hline
    13.2345 & 15.5212 & 20.9524 & 18.0239 & 41.2219 & 21.9963 \\ \hline
    12.7587 & 15.8331 & 20.9524 & 18.3655 & 41.5886 & 21.9963 \\ \hline
    12.255 & 14.7693 & 18.9474 & 16.6416 & 40.7278 & 19.9963 \\ \hline
    15.5824 & 16.9586 & 25.9615 & 20.944 & 41.7586 & 27.9952 \\ \hline
    13.0898 & 15.7372 & 20.9524 & 18.2323 & 41.4179 & 21.9963 \\ \hline
    13.6639 & 16.0848 & 22.9565 & 19.4197 & 41.1974 & 23.9963 \\ \hline
   \end{tabular}

\paragraph{}
{\bf Random graphs} obtained by a Barycentric refinement of a random graph 
with 20 vertices and 100 edges:

 \begin{tabular}{|c|c|c|c|c|c|} \hline
     $\rho$ & $|\rho|$ & Thm~(\ref{dualvertex}) & 3Walk & \cite{BrualdiHoffmann,Stanley1987} & \cite{LiShiuChan2010} \\ \hline
    15.099 & 15.099 & 16.9412 & 16.1703 & 54.2612 & 27.9994 \\ \hline
    14.2081 & 14.2081 & 15.9375 & 15.2196 & 53.9441 & 25.9994 \\ \hline
    15.0969 & 15.0969 & 16.9412 & 16.1299 & 54.0002 & 27.9994 \\ \hline
    15.1967 & 15.1967 & 16.9412 & 16.2192 & 54.798 & 27.9994 \\ \hline
    15.0951 & 15.0951 & 16.9412 & 16.1226 & 54.2984 & 27.9994 \\ \hline
    15.1091 & 15.1091 & 16.9412 & 16.1761 & 54.4654 & 27.9994 \\ \hline
    16.0813 & 16.0813 & 17.9444 & 17.0785 & 54.0189 & 29.9994 \\ \hline
   \end{tabular}

\paragraph{}
Here are some computations illustrated graphically.

\begin{center} \begin{figure}
\scalebox{0.42}{\includegraphics{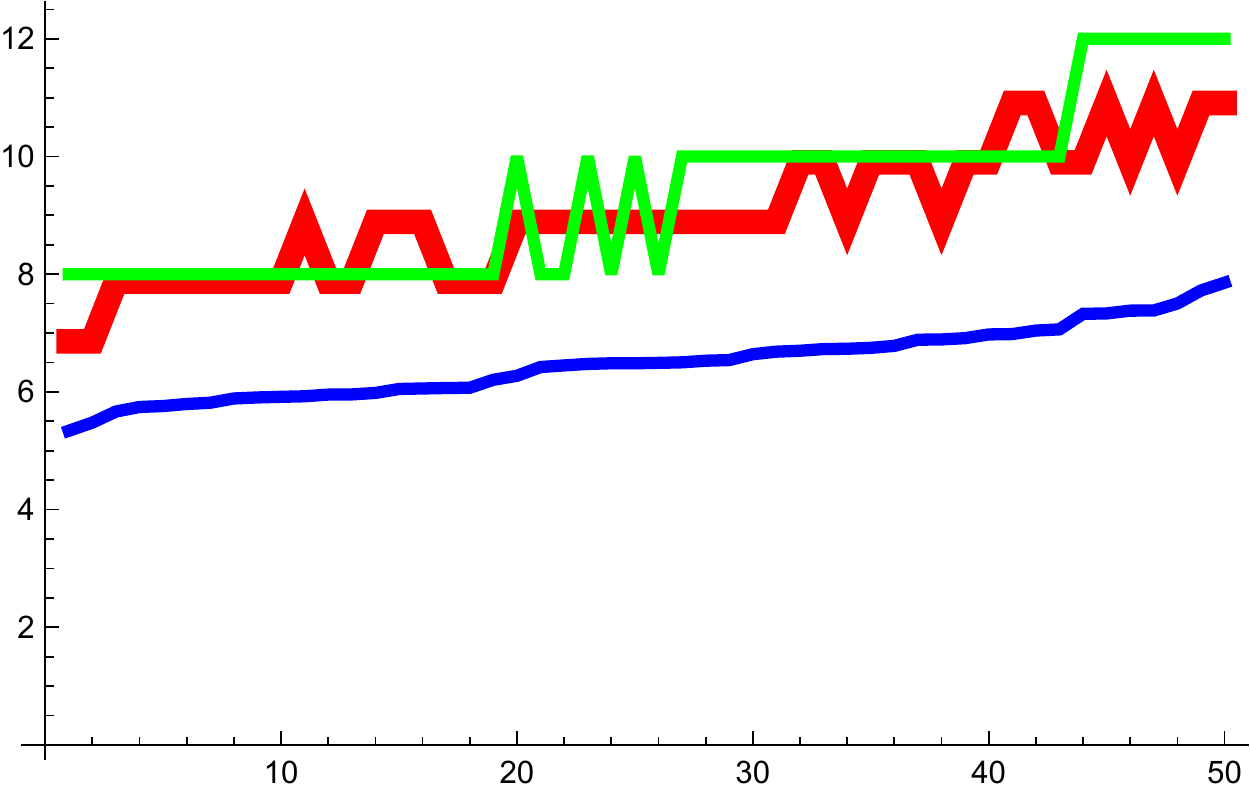}} 
\scalebox{0.42}{\includegraphics{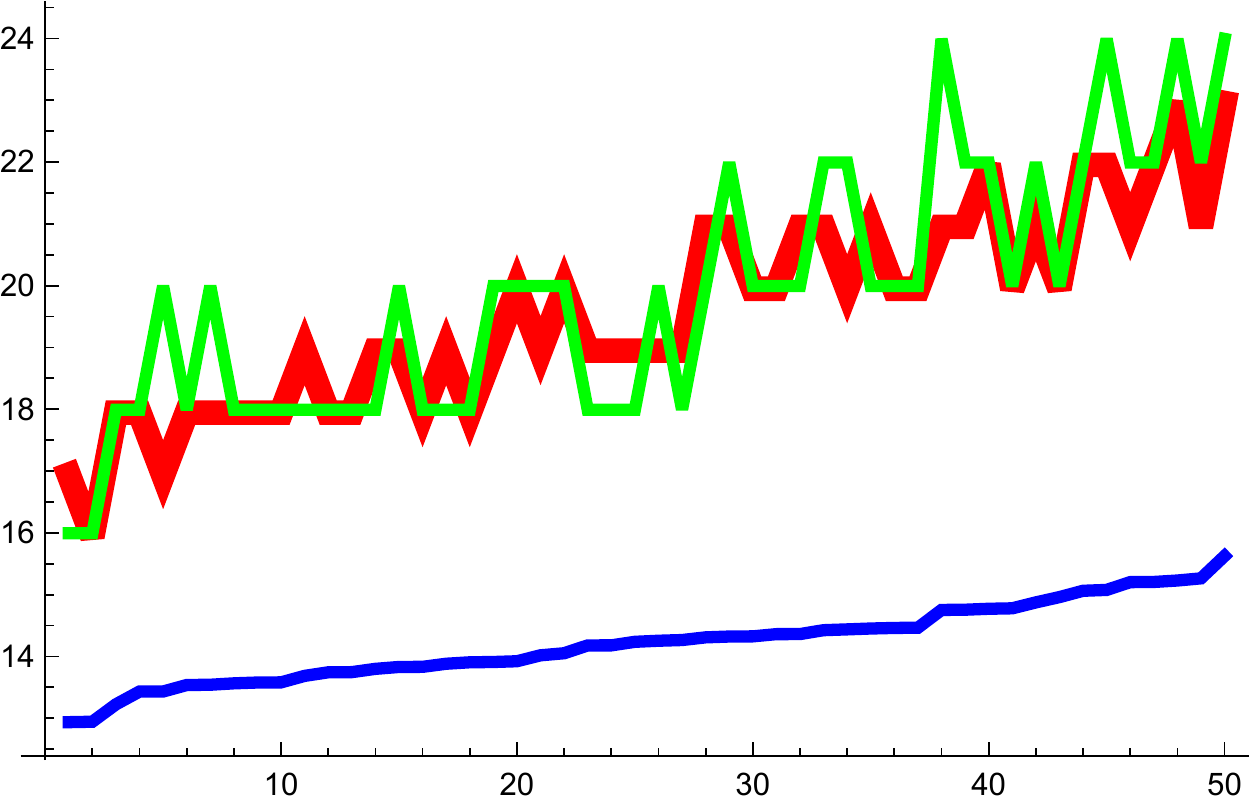}}  \\
\scalebox{0.42}{\includegraphics{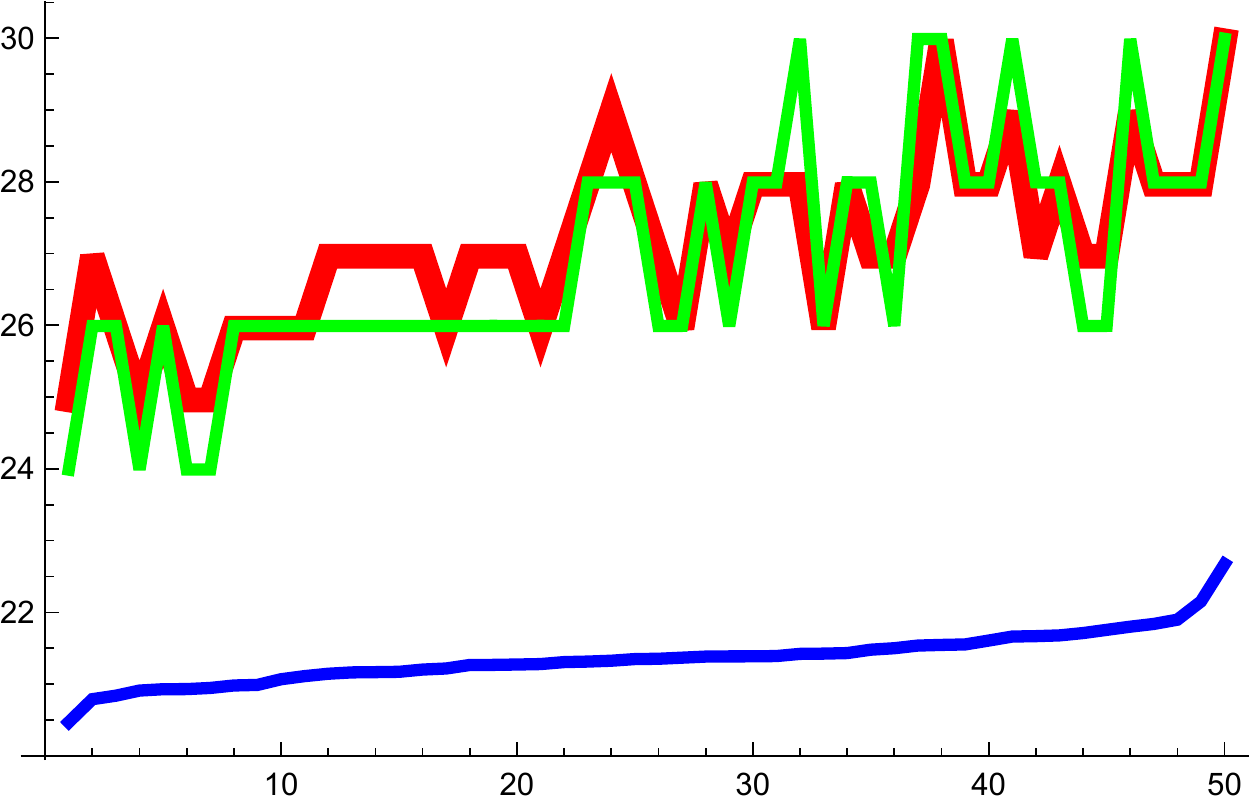}} 
\scalebox{0.42}{\includegraphics{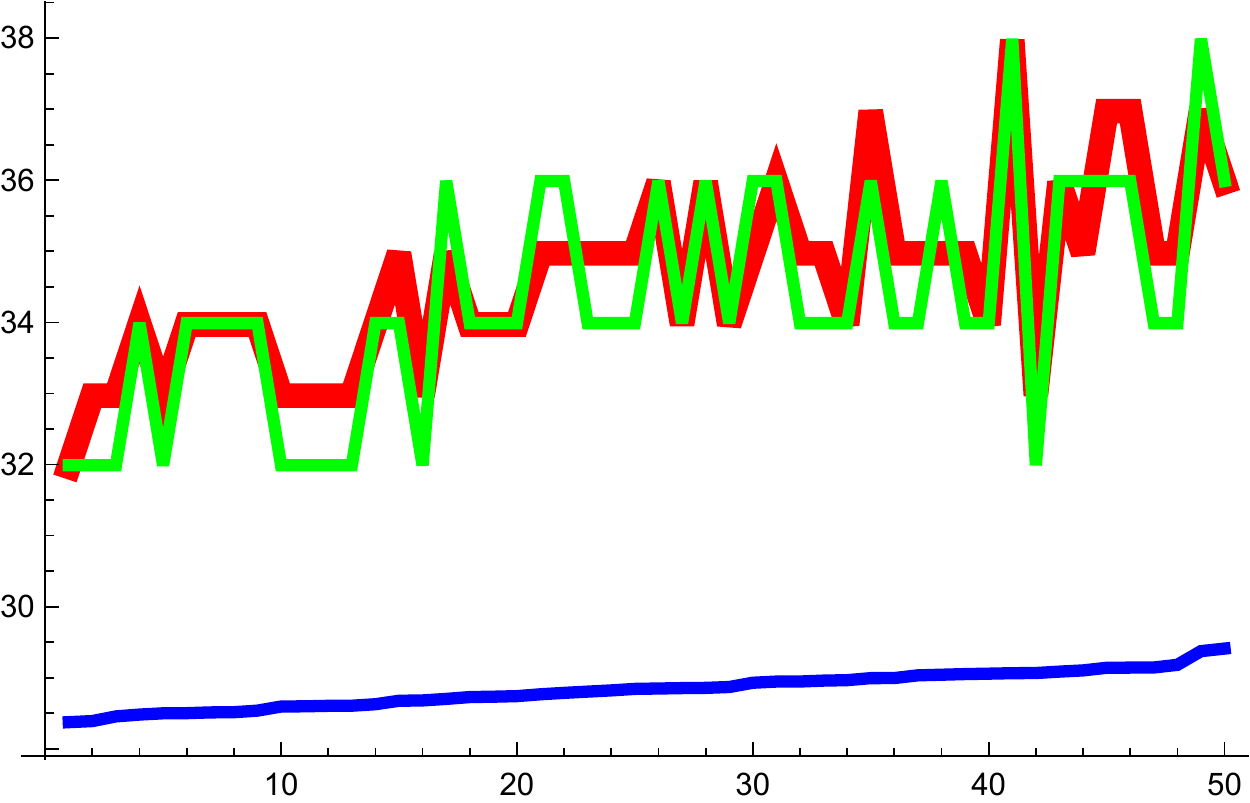}} 
\caption{
\label{eigenvectors}
We see measurements with random graphs from $E(20,p)$ with $p=0.1,0.3,0,5,0.7$. 
50 experiments were done in each case. The lowest curve is $\rho(G)$.
The thin green is the estimate $2d-Q$, with $Q$ small, estimated by many 
authors. The red (thicker) curve is the dual vertex degree estimate given 
in (\ref{dualvertex}). It is more effective for small $p$. 
}
\end{figure} \end{center}

\begin{center} \begin{figure}
\scalebox{0.42}{\includegraphics{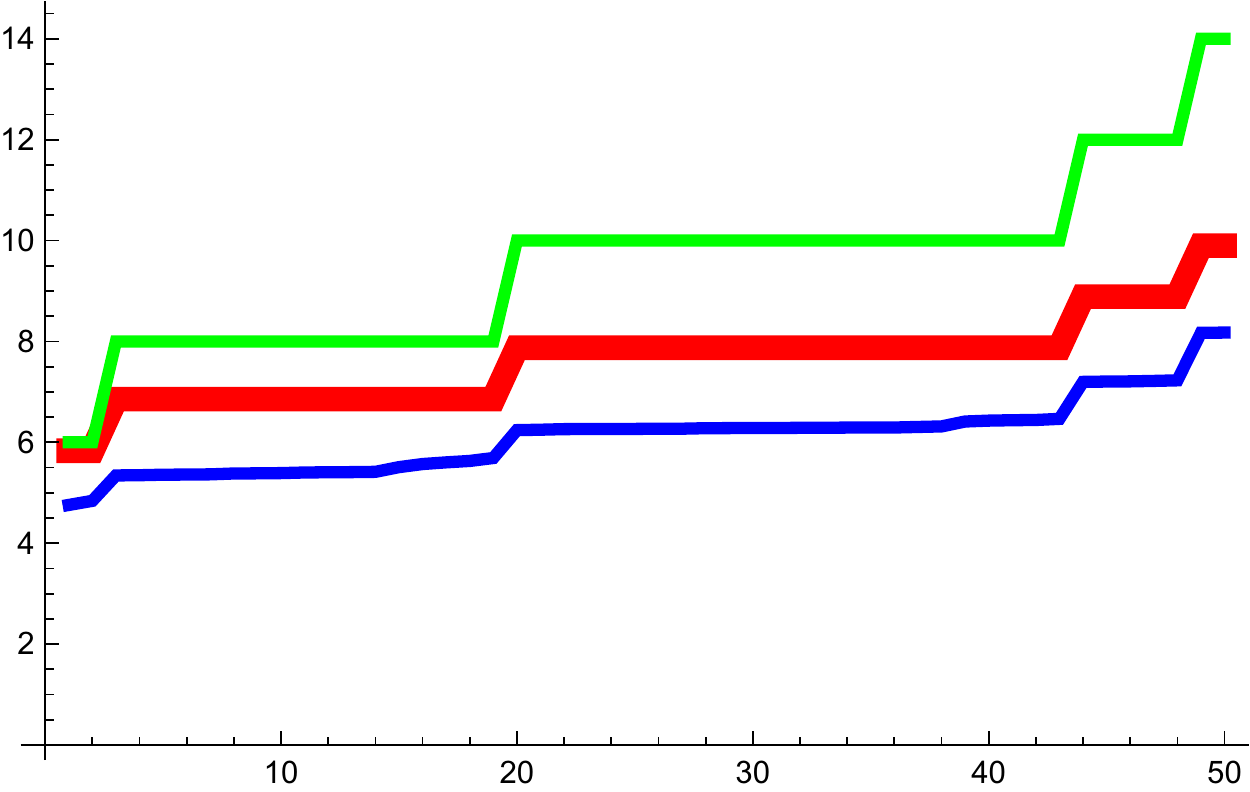}}
\scalebox{0.42}{\includegraphics{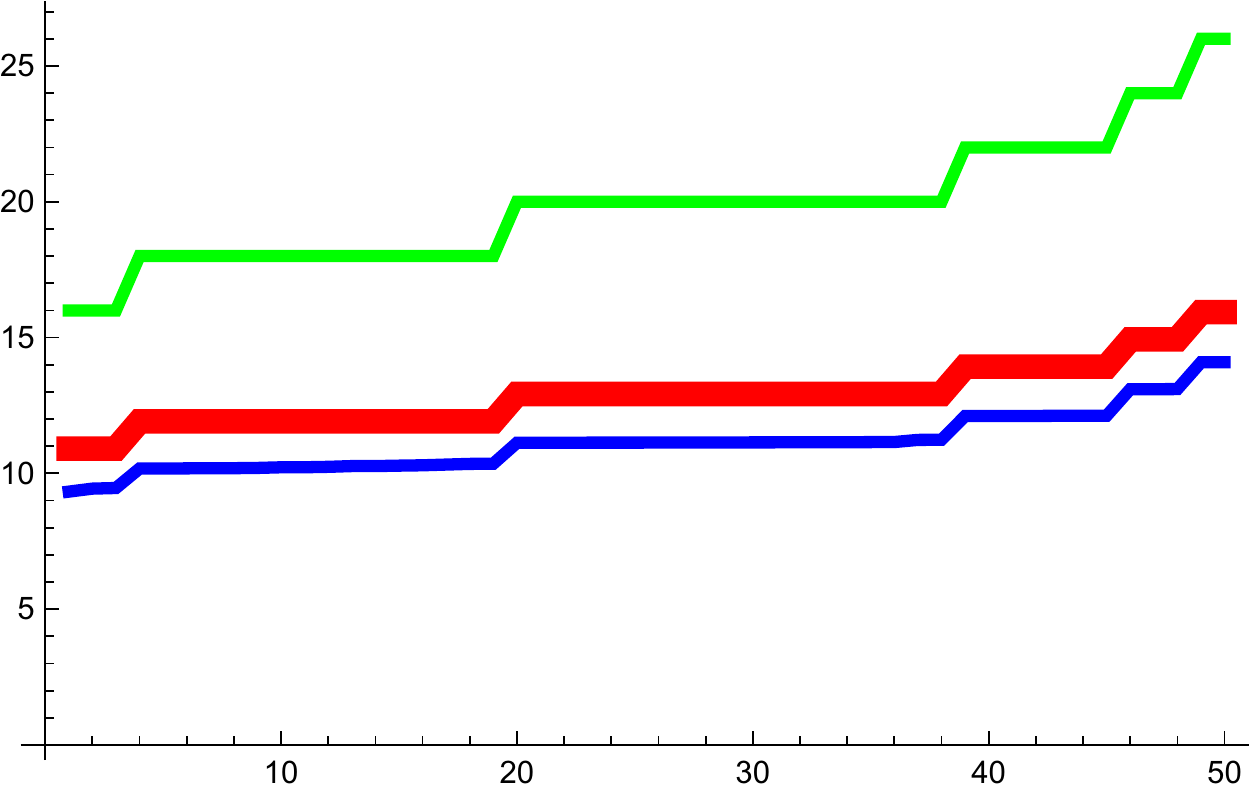}}  \\
\scalebox{0.42}{\includegraphics{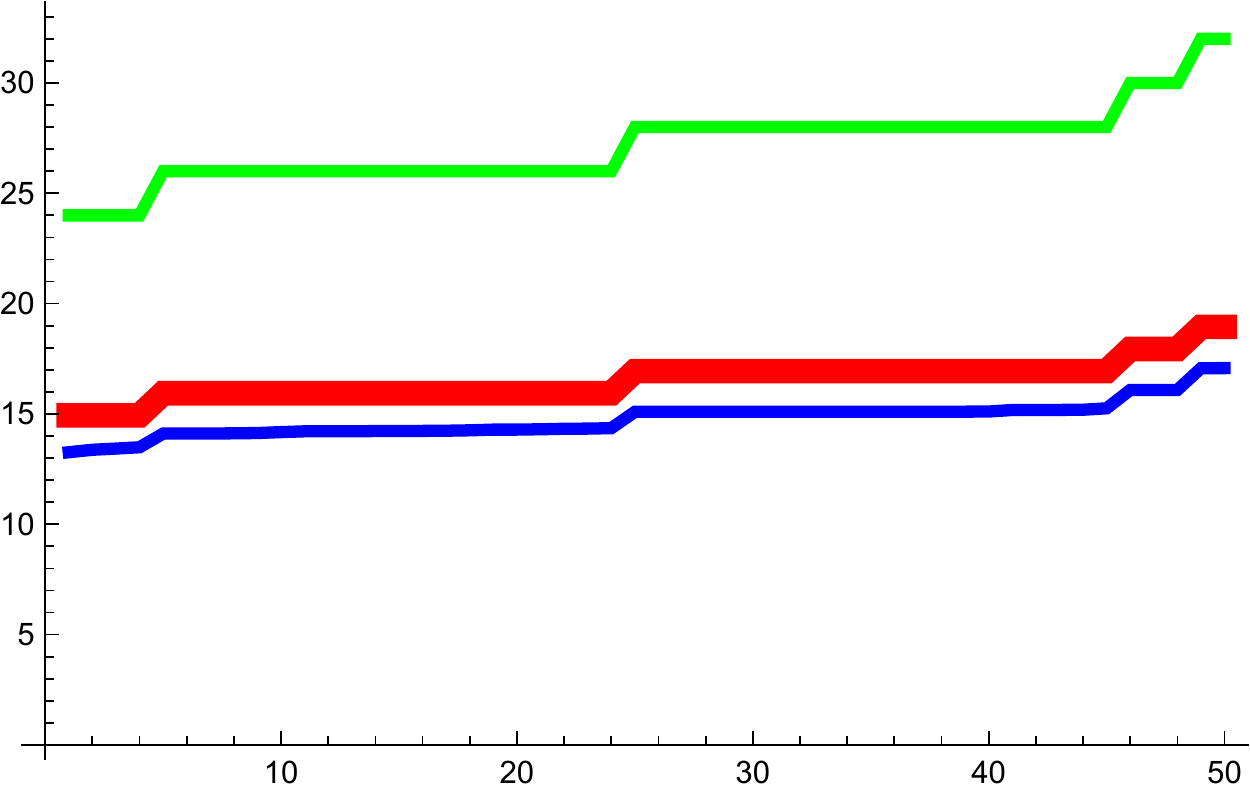}}
\scalebox{0.42}{\includegraphics{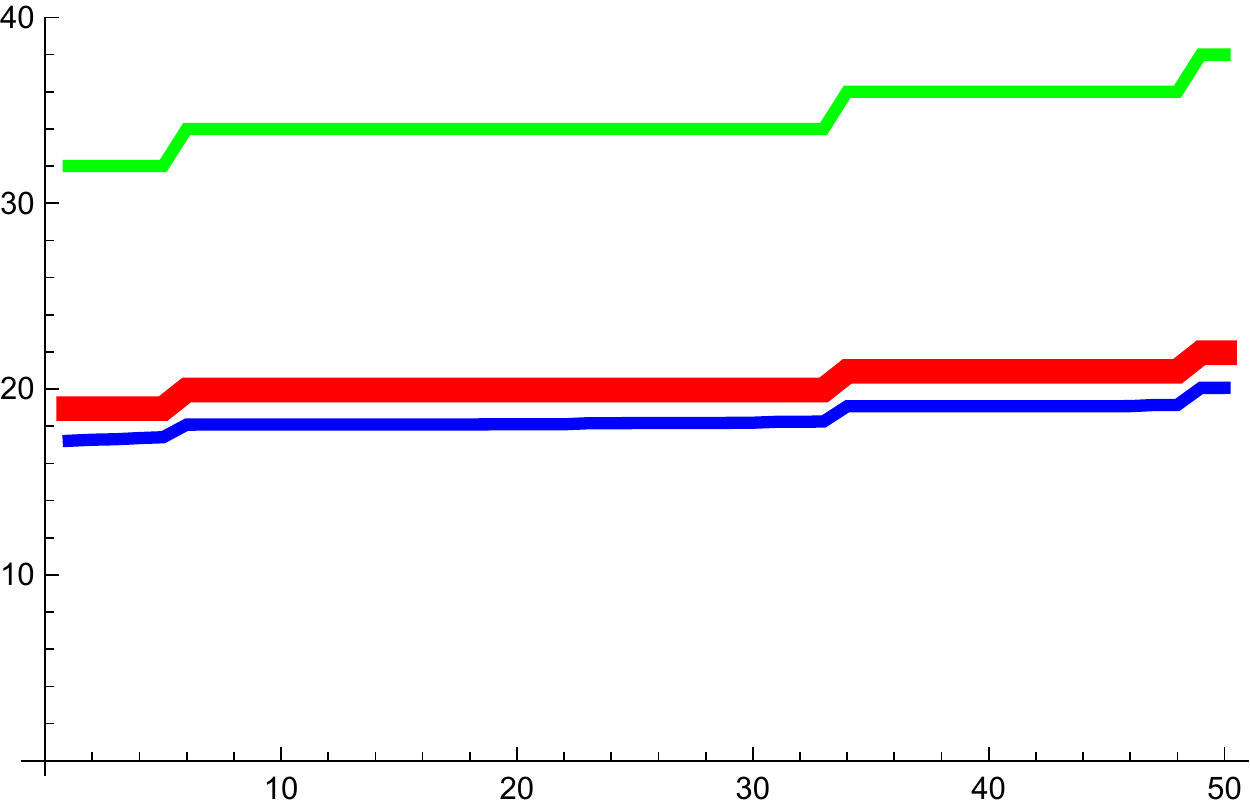}}
\caption{
\label{eigenvectors}
Measurements with random Barycentric refined graphs from $E(20,p)$ 
with $p=0.1,0.3,0,5,0.7$ are displayed. One can see again the estimate 
$2d-Q$ and then the dual vertex degree estimate.
}
\end{figure} \end{center}

\begin{center} \begin{figure}
\scalebox{0.42}{\includegraphics{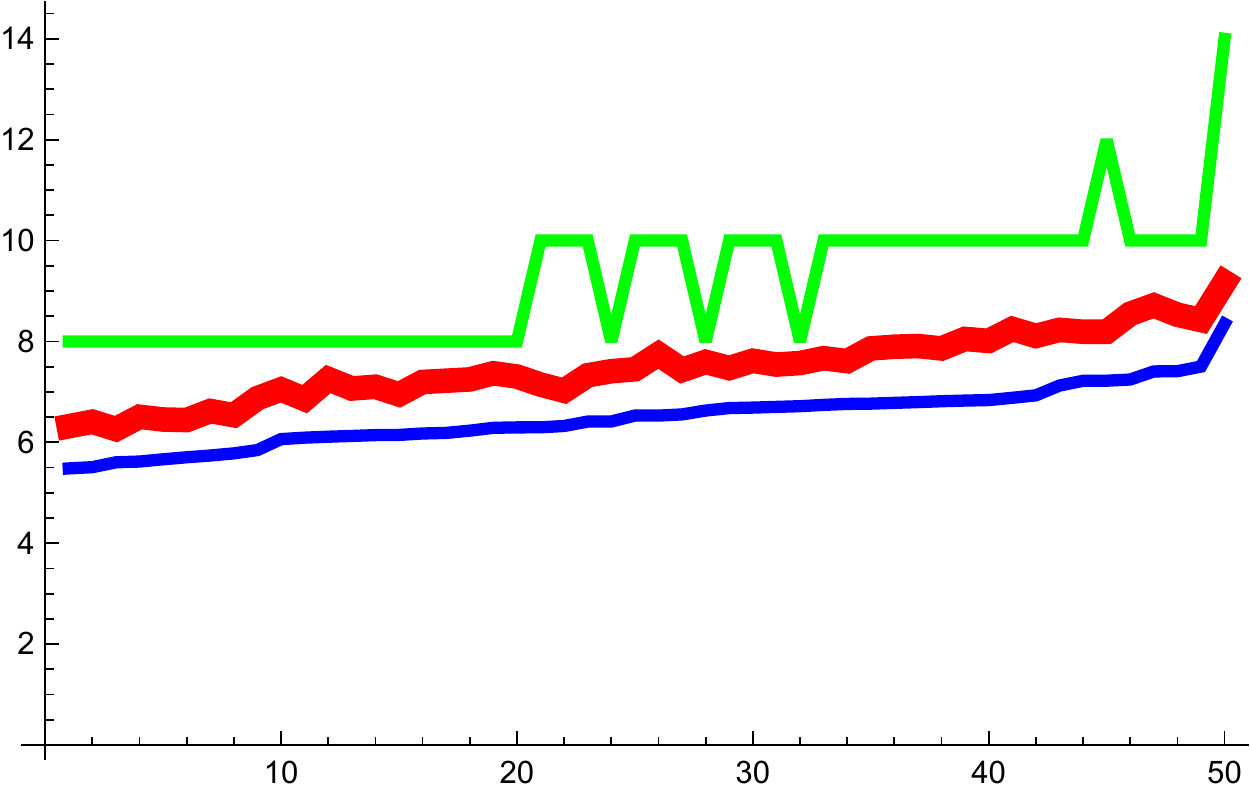}}
\scalebox{0.42}{\includegraphics{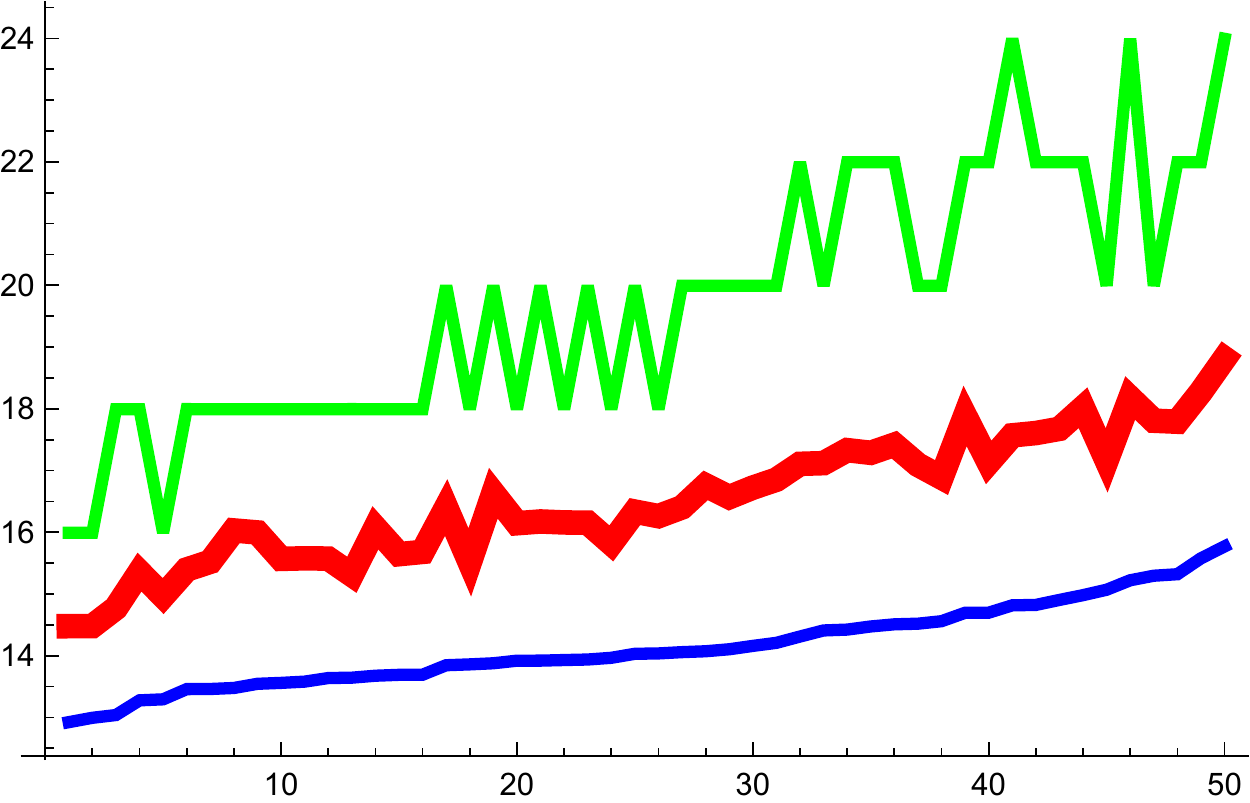}}  \\
\scalebox{0.42}{\includegraphics{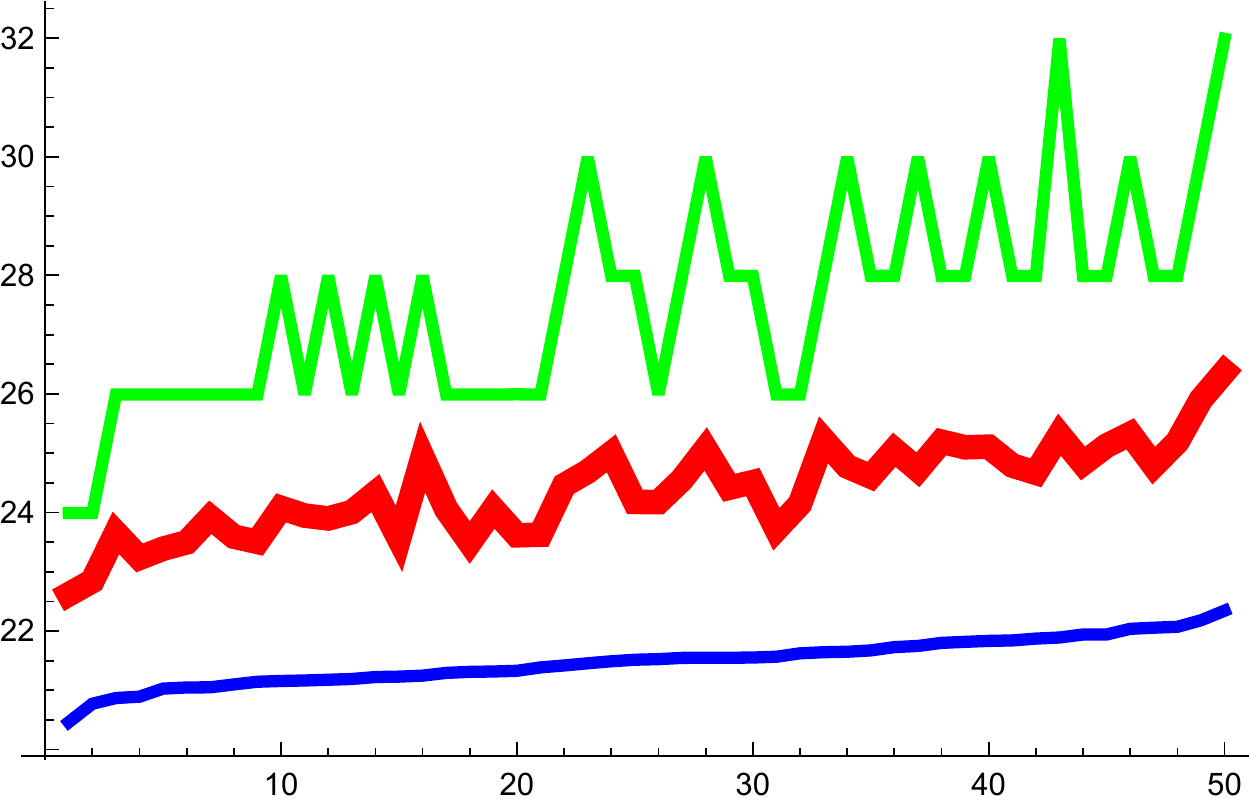}}
\scalebox{0.42}{\includegraphics{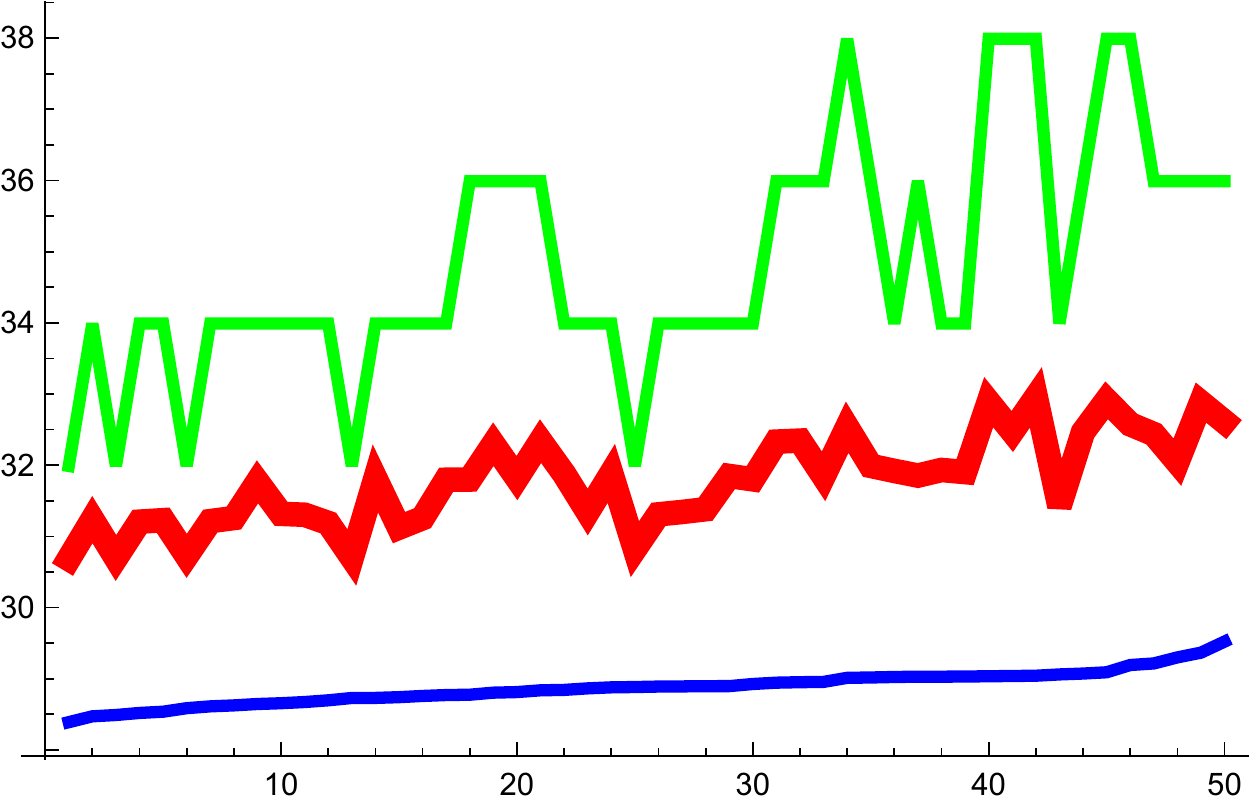}}
\caption{
\label{eigenvectors}
Measurements with random graphs from $E(20,p)$ are seen
for $p=0.1,0.3,0,5,0.7$. The thick estimate
is obtained by counting the number of walks of length $3$ in 
the connection graph $G'$ and taking the maximum. It is
better than $2d-Q$ but it also uses a larger neighborhood of
a vertex. }
\end{figure} \end{center}

\bibliographystyle{plain}

\begin{thebibliography}{10}

\bibitem{AndersonMorely1985}
W.N. Anderson and T.D. Morley.
\newblock Eigenvalues of the {L}aplacian of a graph.
\newblock {\em Linear and Multilinear Algebra}, 18(2):141--145, 1985.

\bibitem{BrualdiHoffmann}
R.~A. Brualdi and A.~J. Hoffmann.
\newblock On the spectral radius of (0,1)-matrices.
\newblock {\em Linear Algebra and its Applications}, 65:133--146, 1985.

\bibitem{Cycon}
H.L. Cycon, R.G.Froese, W.Kirsch, and B.Simon.
\newblock {\em {Schr\"odinger} Operators---with Application to Quantum
  Mechanics and Global Geometry}.
\newblock Springer-Verlag, 1987.

\bibitem{Das2004}
K.Ch. Das.
\newblock The {L}aplacian spectrum of a graph.
\newblock {\em Comput. Math. Appl.}, 48(5-6):715--724, 2004.

\bibitem{VerdiereGraphSpectra}
Y.Colin de~Verdi{\`e}re.
\newblock {\em Spectres de Graphes}.
\newblock Soci{\'e}te Math{\'e}matique de France, 1998.

\bibitem{FengLiZhang}
L.~Feng, Q.~Li, and X-D. Zhang.
\newblock Some sharp upper bounds on the spectral radius of graphs.
\newblock {\em Taiwanese J. Math.}, 11(4):989--997, 2007.

\bibitem{GroneMerrisSunder2}
R.~Grone and R.~Merris.
\newblock The {L}aplacian spectrum of a graph. {II}.
\newblock {\em SIAM J. Discrete Math.}, 7(2):221--229, 1994.

\bibitem{Guo2005}
Ji-Ming Guo.
\newblock A new upper bound for the {L}aplacian spectral radius of graphs.
\newblock {\em Linear Algebra and its Applications}, 400:61--66, 2005.

\bibitem{Hed69}
G.A. Hedlund.
\newblock Endomorphisms and automorphisms of the shift dynamical system.
\newblock {\em Math. Syst. Theor.}, 3:320--375, 1969.

\bibitem{HofKnill}
A.~Hof and O.~Knill.
\newblock Cellular automata with almost periodic initial conditions.
\newblock {\em Nonlinearity}, 8(4):477--491, 1995.

\bibitem{CausalDynamicalTriangulation}
J.~Jurkiewicz, R.~Loll, and J.~Ambjorn.
\newblock Using causality to solve the puzzle of quantum spacetime.
\newblock {\em Scientific American}, 2008.

\bibitem{Kni98}
O.~Knill.
\newblock A remark on quantum dynamics.
\newblock {\em Helvetica Physica Acta}, 71:233--241, 1998.

\bibitem{DehnSommerville}
O.~Knill.
\newblock On a {D}ehn-{S}ommerville functional for simplicial complexes.
\newblock {\\}https://arxiv.org/abs/1705.10439, 2017.

\bibitem{HearingEulerCharacteristic}
O.~Knill.
\newblock One can hear the {E}uler characteristic of a simplicial complex.
\newblock {\\}https://arxiv.org/abs/1711.09527, 2017.

\bibitem{StrongRing}
O.~Knill.
\newblock The strong ring of simplicial complexes.
\newblock {\\}https://arxiv.org/abs/1708.01778, 2017.

\bibitem{DyadicRiemann}
O.~Knill.
\newblock An elementary {D}yadic {R}iemann hypothesis.
\newblock {\\}https://arxiv.org/abs/1801.04639, 2018.

\bibitem{ListeningCohomology}
O.~Knill.
\newblock Listening to the cohomology of graphs.
\newblock {\\}https://arxiv.org/abs/1802.01238, 2018.

\bibitem{KrivelevichSudakov}
M.~Krivelevich and B~Sudakov.
\newblock The largest eigenvalue of sparse random graphs.
\newblock {\em Combinatorics, Probability and Computing}, 12:61--72, 2003.

\bibitem{ShiuChan2009}
J.~Li, W.C. Shiu, and W.H. Chan.
\newblock The {L}aplacian spectral radius of some graphs.
\newblock {\em Linear algebra and its applications}, pages 99--103, 2009.

\bibitem{LiShiuChan2010}
J.~Li, W.C. Shiu, and W.H. Chan.
\newblock The {L}aplacian spectral radius of graphs.
\newblock {\em Czechoslovak Mathematical Journal}, 60:835--847, 2010.

\bibitem{LiZhang1997}
J.S. Li and X.D. Zhang.
\newblock A new upper bound for eigenvalues of the {L}aplacian matrix of a
  graph.
\newblock {\em Linear Algebra and Applications}, 265:93--100, 1997.

\bibitem{MincNonnegative}
H.~Minc.
\newblock {\em Nonnegative Matrices}.
\newblock John Wiley and Sons, 1988.

\bibitem{Pastur}
L.~Pastur and A.Figotin.
\newblock {\em Spectra of Random and Almost-Periodic Operators}, volume 297.
\newblock Springer-Verlag, Berlin--New York, {Grundlehren} der mathematischen
  {Wissenschaften} edition, 1992.

\bibitem{Carmona}
J.Lacroix R.~Carmona.
\newblock {\em Spectral Theory of Random {S}chr\"odinger Operators}.
\newblock Birkh\"auser, 1990.

\bibitem{GroneMerrisSunder1}
R.~Merris R.~Grone and V.S. Sunder.
\newblock The {L}aplacian spectrum of a graph.
\newblock {\em SIAM J. Matrix Anal. Appl.}, 11(2):218--238, 1990.

\bibitem{Rivin}
I.~Rivin.
\newblock Walks on groups, conuting reducible matrices, polynomials and surface
  and free group automorphisms.
\newblock {\em Duke Math. J.}, 142:353--379, 2008.

\bibitem{Shi2007}
L.~Shi.
\newblock Bounds of the {L}aplacian spectral radius of graphs.
\newblock {\em Linear algebra and its applications}, pages 755--770, 2007.

\bibitem{Stanley1987}
R.~P. Stanley.
\newblock A bound on the spectral radius of graphs.
\newblock {\em Linear Algebra and its Applications}, 87:267--269, 1987.

\bibitem{Stevanovic}
D.~Stevanovic.
\newblock {\em Spectral Radius of Graphs}.
\newblock Elsevier, 2015.

\bibitem{Zhang2004}
Xiao-Dong Zhang.
\newblock Two sharp upper bounds for the {L}aplacian eigenvalues.
\newblock {\em Linear Algebra and Applications}, 376:207--213, 2004.

\bibitem{ZhouXu}
H.~Zhou and X.~Xu.
\newblock Sharp upper bounds for the {L}aplacian spectral radius of graphs.
\newblock {\em Mathematical Problems in Engineering}, 2013, 2013.

\end{thebibliography}

\end{document}